\documentclass[final,3p]{elsarticle}
\journal{Linear Algebra and its Applications}

\usepackage{amsfonts}
\usepackage{amsmath}
\usepackage{amssymb}
\usepackage{amsthm}

\graphicspath{{figures/}}

\usepackage{enumerate}
\usepackage{verbatim}

\usepackage{
	float,		
	subcaption, 
	xfrac,		
	colortbl,
	mathtools  	
}

\usepackage[symbol]{footmisc}

\usepackage{hyperref}
\usepackage{todonotes}

\usepackage{pgfplots}
\pgfplotsset{compat=1.8}
\usepackage{tikz}
\usepackage{tikz-cd}

\newcommand{\cond}{\mathrm{cond}}

\newtheorem{theorem}{Theorem}[section]

\newtheorem{remark}[theorem]{Remark}

\newtheorem{definition}[theorem]{Definition}
\newtheorem{example}[theorem]{Example}

\newtheorem{alg}[theorem]{Algorithm}

\newenvironment{Remark}{\goodbreak \begin{remark}\slshape}{\end{remark}}
\newenvironment{Example}{\goodbreak \begin{example}\slshape}{\end{example}}

\makeatletter
\def\imod#1{\allowbreak\mkern10mu({\operator@font mod}\,\,#1)}
\makeatother

\makeatletter 
\let\c@algorithm\c@theorem  
\makeatother

\newenvironment{algorithm}[1]{\goodbreak~\begin{alg}[#1]~\vspace{-9pt}~\\
		\rule{\linewidth}{0.5pt}~\\}{\vspace{-9pt}~\\
		\rule{\linewidth}{0.5pt}~\end{alg}}

\numberwithin{equation}{section}
\numberwithin{table}{section}
\numberwithin{figure}{section}

\newcommand{\e}{\mathrm e}
\renewcommand{\i}{\mathrm i}

\renewcommand{\b}{\boldsymbol} 

\newcommand{\R}{\mathbb R}
\newcommand{\C}{\mathbb C}
\newcommand{\Z}{\mathbb Z}
\newcommand{\N}{\mathbb N}
\newcommand{\T}{\mathbb T}
\renewcommand{\H}{\mathcal H}

\newcommand{\ex}{\hspace*{0ex} \hfill \hbox{\vrule height
		1.5ex\vbox{\hrule width 1.4ex \vskip 1.4ex\hrule  width 1.4ex}\vrule
		height 1.5ex}}

\long\def\symbolfootnote[#1]#2{\begingroup%
	\def\thefootnote{\fnsymbol{footnote}}\footnote[#1]{#2}\endgroup}

\hypersetup{pdfauthor={Melanie Kircheis, Daniel Potts},
	pdftitle={Direct inversion of the nonequispaced fast Fourier transform},
	pdfsubject={},
	pdfcreator = {pdflatex},
	plainpages=false,
	pdfstartview=FitH,       
	pdfview=FitH,            
	pdfpagemode=UseOutlines, 
	bookmarksnumbered=true, 
	bookmarksopen=false,     
	bookmarksopenlevel=0,   
	colorlinks=true,       
	linkcolor=black,
	citecolor=black,
	urlcolor=black}

\begin{document}	

\begin{frontmatter}

\title{Direct inversion of the nonequispaced fast Fourier transform}
\author{Melanie Kircheis}\ead{melanie.kircheis@math.tu-chemnitz.de}
\author{Daniel Potts}\ead{potts@math.tu-chemnitz.de}
\address{Technische Universit\"at Chemnitz, Faculty of Mathematics, 09107 Chemnitz, Germany}

\begin{abstract}
Various applications such as MRI, solution of PDEs, etc. need to perform an inverse nonequispaced fast Fourier transform (NFFT), i.\,e., compute $M$ Fourier coefficients from given $N$ nonequispaced data.
In the present paper we consider direct methods for the inversion of the NFFT.
We introduce algorithms for the setting $M=N$ as well as for the underdetermined and over\-determined cases.
For the setting $M=N$ a direct method of complexity $\mathcal O(N\log N)$ is presented, which utilizes Lagrange interpolation and the fast summation.
For the remaining cases, we use the matrix representation of the NFFT to deduce our algorithms.
Thereby, we are able to compute an inverse NFFT up to a certain accuracy by means of a modified adjoint NFFT in $\mathcal O(M\log M+N)$ arithmetic operations.
Finally, we show that these approaches can also be explained by means of frame approximation.
\end{abstract}


\end{frontmatter}


\section{Introduction}

The NFFT, short hand for nonequispaced fast Fourier transform or nonuniform fast Fourier transform (NUFFT), respectively, is a fast algorithm to evaluate a trigonometric polynomial
\begin{equation}
\label{eq:trig_polynomial}
	f(x) = \sum_{k = -\frac M2}^{\frac M2 -1} \hat f_k\, \e^{2\pi\i k x}
\end{equation}
for given Fourier coefficients \mbox{$\hat f_k\in\C$}, \mbox{$k=-\frac M2, \dots, \frac M2-1$}, \mbox{$M\in 2\N$}, at nonequispaced points\mbox{ $x_j \in \left[-\frac 12,\frac 12\right)$}, \mbox{$\,j=1,\dots,N$}.
In case we are given equi\-spaced points and $M=N$, this evaluation can be realized by means of the fast Fourier transform (FFT).
For this setting also an algorithm for the inverse problem is known.
Hence, we are interested in an inversion also for non\-equi\-spaced data, i.\,e., the Fourier coefficients~$\hat f_k$ shall be computed for given function values~$f_j=f(x_j)$ of the trigonometric polynomial~\eqref{eq:trig_polynomial}.
Additionally, we study the inversion of the adjoint problem, namely the reconstruction of function values~\mbox{$f_j\in\C$} from given data
\begin{equation}
\label{eq:nfft*}
	h_k = \sum_{j=1}^N f_j\, \e^{-2\pi\i kx_j}, \quad k=-\tfrac M2, \dots, \tfrac M2 - 1.
\end{equation}

In general, the number~$N$ of nodes~$x_j$ is independent from the number~$M$ of Fourier coefficients~$\hat f_k$ and therefore the nonequispaced Fourier matrix
\begin{equation}
\label{eq:matrix_A}
	\b A \coloneqq \left(\e^{2\pi\i k x_j}\right)_{j=1,\, k=-\frac M2}^{N,\; \frac M2 -1}\ \in \mathbb C^{N\times M},
\end{equation}
which we would have to invert, is rectangular in most cases.
Nevertheless, several approaches have been developed to compute an inverse NFFT (iNFFT).
First of all, there are some iterative methods.
Recently, in \cite{RuTo18} an algorithm was published for the setting $M=N$ which is based on the CG method as well as low rank approximation and is specially designed for jittered equispaced points.
An approach for the overdetermined case can be found in \cite{FeGrSt95}, where the solution is computed iteratively by means of the CG algorithm using the Toeplitz structure of $\b A^* \b W \b A$ with a diagonal matrix~$\b W$ of Voronoi weights.
In \cite{kupo04} the CG method in connection with the NFFT was used to formulate an iterative algorithm for the underdetermined setting which deploys $\b A \b{\hat W} \b A^*$ with weights $\b{\hat W}$ based on kernel approximation.
Furthermore, already in \cite{duro95} a direct method was explained for the setting $M=N$ which uses Lagrange interpolation as well as fast multipole methods.
Based on this, in \cite{Selva18} another direct method was deduced for the same setting which also uses Lagrange interpolation but additionally incorporates an imaginary shift for the fast evaluation of occurring sums.
Since $\b A^* \b A$ is a Toeplitz matrix another direct method for the overdetermined setting can be derived using this special structure, see \cite{HeRo84}, analogously to \cite{AvShSh16}.
In addition, also a frame-theoretical approach is known from~\cite{GeSo14} which provides a link between the adjoint NFFT and frame approximation and could therefore be seen as a method to invert the NFFT.

In this paper we present new direct methods for inverting the NFFT in general.
For the quadratic setting, i.\,e., $M=N$, we review our method introduced in \cite{KircheisBA17}, which is also based on Lagrange interpolation but utilizes the fast summation to evaluate occurring sums.
For the general case, we use as a motivation that for equispaced points an inversion can be realized by 
$\b A \b A^* \approx M \b I_N$ 
and aim to generalize this result to find a good approximation of the inversion for nonequispaced nodes.
To this end, we employ the decomposition 
$\b A \approx \b B \b F \b D$
known from the NFFT approach and compute the sparse matrix $\b B$ such that we receive an approximation of the form
$\b A \b D^* \b F^* \b B^* \approx M \b I_N$.
In other words, we are able to compute an inverse NFFT by means of a modified adjoint NFFT.
Analogously, an inverse adjoint NFFT can be obtained by modifying the NFFT.
Hence, the inversions can be computed in \mbox{$\mathcal O(M\log M+N)$} arithmetic operations.
The necessary precomputations developed in this paper are of complexity $\mathcal O(N^2)$ and $\mathcal O(M^2)$, respectively.
Therefore, our method is especially beneficial in case we are given fixed nodes for several problems.
Finally, we show that these approaches can also be explained by means of frame approximation.

The present work is organized as follows.
In Section~\ref{sec:nfft} we introduce the already mentioned algorithm, the NFFT.
Afterwards, in Section~\ref{sec:inversion} we deal with the inversion of this algorithm.
In Section~\ref{subsec:lagrange} we firstly review our method from~\cite{KircheisBA17} for the quadratic setting~$M=N$.
Secondly, in Section~\ref{subsec:rect} the underdetermined and overdetermined settings are studied, which are treated separately in Sections~\ref{subsubsec:ansatz1} and \ref{subsubsec:ansatz2}.
Finally, in Section~\ref{sec:frames} we deduce an approach for the inversion which is based on frame theory.
Therefore, first of all, the main ideas of frames and approximation via frames will be introduced in Section~\ref{subsec:approx} and subsequently, in Section~\ref{subsec:link}, we will use these ideas to develop an approach for the iNFFT adapted from \cite{GeSo14}.
In the end, we will see that this frame-theoretical approach can be traced back to the methods for the inversion introduced in Section~\ref{subsec:rect}.

\section[Nonequispaced fast Fourier transform (NFFT)]{Nonequispaced fast Fourier transform\label{sec:nfft}}

For given nodes \mbox{$x_j \in \left[-\frac 12, \frac 12\right)$},\, \mbox{$j=1,\dots,N$}, \mbox{$M \in 2\N$}, as well as arbitrary Fourier coefficients \mbox{$\hat{f_k} \in \C,$} \mbox{$k = -\frac M2,\dots,\frac M2-1,$} we consider the computation of the sums 
\begin{equation}
\label{eq:nfft}
	f_j = f(x_j) = \sum_{k=-\frac M2}^{\frac M2 -1} \hat{f_k}\, \e^{2\pi\i kx_j}, \quad j=1,\dots,N,
\end{equation}
as well as the adjoint problem of the computation of the sums~\eqref{eq:nfft*} for given values $f_j \in \C$.
A fast algorithm to solve this problem is called \textbf{nonequispaced fast Fourier transform (NFFT)} and is briefly explained below, cf. \cite{duro93, bey95, st97, postta01, post02, GrLe04, KeKuPo09, RuTo18}.

By defining the matrix~\eqref{eq:matrix_A} as well as the vectors
\mbox{$\b f\coloneqq\left(f_j\right)_{j=1}^N$}, 
\mbox{$\b{\hat f}\coloneqq(\hat f_k)_{k = -\frac M2}^{\frac M2 -1}$}
and
\mbox{$\b h \coloneqq (h_k)_{k = -\frac M2}^{\frac M2 -1}$},
the computation of sums of the form \eqref{eq:nfft} and \eqref{eq:nfft*} can be written as
\mbox{$\b f = \b A \b{\hat f}$}
and
\mbox{$\b h = \b A^* \b f$},
where 
\mbox{$\b A^* = \overline{\b A}^{\mathrm T}$} 
denotes the adjoint matrix of~$\b A$.

\subsection{The NFFT\label{subsec:nfft}}

We firstly restrict our attention to problem~\eqref{eq:nfft}, 
which is equivalent to the evaluation of a trigonometric polynomial~$f$ at nodes~$x_j$, see~\eqref{eq:trig_polynomial}.
At first, we approximate~$f$ by a linear combination of translates of a \mbox{1-periodic} function~$\tilde{w}$, i.\,e., 
\begin{equation*}
	f(x) \approx s_1(x) \coloneqq 
	\sum_{l=-\frac{M_\sigma}{2}}^{\frac{M_\sigma}{2}-1} g_l\, \tilde w \hspace{-2pt} \left(x-\tfrac{l}{M_\sigma}\right),
\end{equation*}
where \mbox{$M_\sigma=\sigma M$} with the so-called \textbf{oversampling factor} \mbox{$\sigma \geq 1$}.
In the easiest case $\tilde w$ originates from periodization of a function 
\mbox{$w \colon [-\frac 12, \frac 12) \to \R$}. 
Let this so-called \textbf{window function} be chosen such that its \mbox{1-periodic} version
\mbox{$\tilde w (x) = \sum_{r\in\Z} w(x+r)$}
has an absolutely convergent Fourier series.
By means of the definition
\begin{equation*}
	\hat g_k \coloneqq \sum_{l=-\frac{M_\sigma}{2}}^{\frac{M_\sigma}{2}-1} g_l\, \e^{-2\pi\i kl/{M_\sigma}}, 
	\quad k \in \Z,
\end{equation*}
and the convolution theorem, $s_1$ can be represented as
\begin{align}
	\label{eq:s1}
	s_1(x) 
	&= 
	\sum_{k=-\infty}^{\infty} c_k(s_1)\, \e^{2\pi\i kx}
	\\
	&= 
	\sum_{k=-\frac{M_\sigma}{2}}^{\frac{M_\sigma}{2}-1} \hat g_k\;c_k(\tilde{w})\, \e^{2\pi\i kx} + \sum_{r=-\infty \atop{r\ne 0}}^{\infty} \sum_{k=-\frac{M_\sigma}{2}}^{\frac{M_\sigma}{2}-1} \hat g_k\;c_{k+M_\sigma r}(\tilde{w})\, \e^{2\pi\i (k+M_\sigma r)x}. \notag
\end{align}
Comparing \eqref{eq:nfft} and \eqref{eq:s1} gives rise for the following definition. We set
\begin{equation*}
	\hat g_k\coloneqq
	\left\{
	\begin{array}{cl}
	\dfrac{\hat f_k}{\hat w(k)} &: k\in\{-\frac M2,\dots,\frac M2-1\},\\
	0 &: k\in\{-\frac{M_\sigma}{2},\dots,\frac{M_\sigma}{2}-1\}\setminus\{-\frac M2,\dots,\frac M2-1\},
	\end{array}
	\right.
\end{equation*}
where the Fourier transform of~$w$ is given by
\begin{equation}
\label{eq:fourier_transform}
	\hat w(k) 
	= 
	\int_{-\infty}^{\infty} w(x) \, \e^{-2\pi\i kx} \,\mathrm dx 
	= 
	\int_{-\frac 12}^{\frac 12} \tilde w(x) \, \e^{-2\pi\i kx} \,\mathrm dx 
	= 
	c_k(\tilde w).
\end{equation}
Furthermore, we suppose $w$ is small outside the interval 
\mbox{$\left[-\sfrac{m}{M_\sigma}, \sfrac{m}{M_\sigma}\right]$,\,} \mbox{$m \ll M_\sigma.$} 
Then $w$ can be approximated by 
\mbox{$w_m (x) = \chi _{\left[-\sfrac{m}{M_\sigma}, \sfrac{m}{M_\sigma}\right]} \cdot w(x)$},
which is compactly supported since
\mbox{$\chi _{\left[-\sfrac{m}{M_\sigma}, \sfrac{m}{M_\sigma}\right]}$}
denotes the characteristic function of 
\mbox{$\left[-\sfrac{m}{M_\sigma}, \sfrac{m}{M_\sigma}\right].$} 
Thus, $\tilde w$ can be approximated by the \mbox{1-periodic} function~$\tilde w_m$ with
\begin{equation*}
	\sum_{k\in\Z} \hat w(k) \, \e^{2\pi\i kx}
	=
	\tilde w(x) \approx \tilde w_m (x) 
	= 
	\sum_{r \in \Z} w_m (x+r).
\end{equation*} 
Hence, we obtain the following approximation
\begin{equation*}
\label{eq:truncation}
	f(x_j) 
	\approx 
	s_1(x_j) 
	\approx 
	s(x_j) 
	\coloneqq 
	\sum_{l=-\frac{M_\sigma}{2}}^{\frac{M_\sigma}{2}-1} g_l\, \tilde w_m \hspace{-2.5pt} \left(x_j-\tfrac{l}{M_\sigma}\right) 
	= 
	\sum_{l = \lceil M_\sigma x_j \rceil -m}^{\lfloor M_\sigma x_j\rfloor +m} g_l\, \tilde w_m \hspace{-2.5pt} \left(x_j - \tfrac{l}{M_\sigma}\right),
\end{equation*}
where simplification arises because many summands vanish.
By defining
\begin{itemize}
	\item the diagonal matrix
	\begin{equation}
	\label{eq:matrix_D}
		\b D 
		\coloneqq 
		\text{diag} \left( \frac 1{M_{\sigma}\hat{w}(k)} \right)_{k=-\frac M2}^{\frac M2-1} 
		\ \in \mathbb C^{M\times M},
	\end{equation}
	\item the truncated Fourier matrix
	\begin{equation}
	\label{eq:matrix_F}
		\b F 
		\coloneqq 
		\left( \e^{2\pi\i k \frac l{M_{\sigma}}} \right)_{l=-\frac{M_{\sigma}}{2}, \,k=-\frac M2}^{\frac{M_{\sigma}}{2} -1, \; \frac M2-1} 
		\ \in \mathbb C ^{M_{\sigma}\times M},
	\end{equation}
	\item and the sparse matrix
	\begin{equation}
	\label{eq:matrix_B}
		\b B 
		\coloneqq 
		\bigg( \tilde w_m \hspace{-2.5pt} \left(x_j - \tfrac l{M_{\sigma}}\right) \bigg)_{j=1,\, l=-\frac{M_{\sigma}}{2}}^{N,\; \frac{M_{\sigma}}{2}-1} 
		\ \in \mathbb R^{N\times M_{\sigma}},
	\end{equation}
\end{itemize} 
this can be formulated in matrix-vector notation and we receive the approximation 
\mbox{$\b A \approx \b B \b F \b D$}.
Therefore, the corresponding fast algorithm consisting of three steps is of complexity \mbox{$\mathcal O(M\log M+N)$}.

\begin{Remark} 
\label{remark:window_functions}
Suitable window functions can be found in \cite{duro93, bey95, st97, DuSc, Fou02, GrLe04, KeKuPo09}.
\ex
\end{Remark}

\begin{Remark}
It must be pointed out that because of consistency the factor~$\frac{1}{M_\sigma}$ is here not located in the matrix~$\b F$ as usual but in the matrix~$\b D$.
\ex
\end{Remark}

\subsection{The adjoint NFFT\label{subsec:nfft*}}

Now we consider the problem~\eqref{eq:nfft*}, which is treated similarly to~\cite{postta01}, and therefore we firstly define the function
\begin{equation}
\label{eq:function_g_tilde}
	\tilde g(x) \coloneqq \sum_{j=1}^{N} f_j \, \tilde w(x_j-x)
\end{equation}
and calculate its Fourier coefficients
\begin{align*}
	c_k(\tilde g) 
	&=
	\int_{-\frac 12}^{\frac 12} \tilde g(x) \, \e^{-2\pi\i kx} \, \mathrm dx 
	=
	\sum_{j=1}^{N} f_j \, \e^{-2\pi\i kx_j} \int_{-\frac 12}^{\frac 12} \tilde w(y) \, \e^{2\pi\i ky} \, \mathrm dy
	= 
	h_k \, c_{-k}(\tilde w).
\end{align*}
In other words, the values~$h_k$ can be computed if $c_{-k}(\tilde w)$ and $c_k(\tilde g)$ are known.
The Fourier coefficients of $\tilde g$ are determined approximately by means of the trapezoidal rule
\begin{equation*}
	c_k(\tilde g) 
	\approx 
	\frac{1}{M_\sigma} \sum_{l=-\frac{M_\sigma}{2}}^{\frac{M_\sigma}{2}-1} 
	\sum_{j=1}^{N} f_j \, \tilde w \hspace{-1.5pt} \left( x_j - \tfrac{l}{M_\sigma} \right) \, \e^{-2\pi\i k l/M_\sigma}.
\end{equation*}
Let the function~$w$ moreover be well localized in time so that $\tilde w$ can be replaced by~$\tilde w_m$ again.
Then we obtain the approximation 
\begin{equation}
\label{eq:approx_nfft*}
	\frac{c_k(\tilde g)}{c_{-k}(\tilde w)} 
	\approx
	\frac{1}{M_\sigma \hat w(-k)} 
	\sum_{l=-\frac{M_{\sigma}}{2}}^{\frac{M_{\sigma}}{2}-1} 
	\sum_{j=1}^{N} 
	f_j\, \tilde w_m \hspace{-2.5pt} \left(x_j-\tfrac{l}{M_{\sigma}}\right) \, \e^{-2\pi\i kl/M_{\sigma}}
	\eqqcolon 
	\tilde h_k.
\end{equation}
Rewriting this by means of \eqref{eq:matrix_D}, \eqref{eq:matrix_F} and \eqref{eq:matrix_B} 
we receive 
\mbox{$\b A^* \approx \b D^* \b F^* \b B^*$}.
Hence, the algorithm for the adjoint problem is also of complexity \mbox{$\mathcal O(M\log M+N)$}.

\section{Inversion of the NFFT\label{sec:inversion}}

Having introduced the fast methods for nonequispaced data, we aim to find an inversion for these algorithms
encouraged by the fact that for equispaced data the inversion is well-known.
Therefore, we face the following two problems.
\begin{enumerate}
	\item[(1)] 
	Solve
	\begin{equation}
	\label{eq:problem_infft}
		\begin{split}
		\b A \b{\hat f} &= \b f, \\
		\text{ given: } \b f \in \C^N,& \text{ find: } \b{\hat f} \in \C^M, 
		\end{split}
	\end{equation}
	i.\,e., reconstruct the Fourier coefficients \mbox{$\b{\hat f} = (\hat f_k)_{k=-\frac M2}^{\frac M2-1}$} from given function values \mbox{$\b f = (f_j)_{j=1}^{N}$}.
	This will be solved by an \textbf{inverse NFFT}.
	\item[(2)]
	Solve
	\begin{equation}
	\label{eq:problem_infft*}
		\begin{split}
		\b A^* \b f &= \b h, \\
		\text{ given: } \b h \in \C^M,& \text{ find: } \b f \in \C^N,
		\end{split}
	\end{equation}
	i.\,e., reconstruct the coefficients \mbox{$\b f = (f_j)_{j=1}^{N}$} from given data \mbox{$\b h = (h_k)_{k=-\frac M2}^{\frac M2-1}$}.
	This will be solved by an \textbf{inverse adjoint NFFT}.
\end{enumerate}
In both problems the numbers $M$ and~$N$ are independent.
It is obvious that except for the quadratic setting \mbox{$M=N$} there are two different ways to choose $M$ and~$N$.
The first possibility is \mbox{$M<N$}, i.\,e., for the inverse NFFT in~\eqref{eq:problem_infft} we are given more function values than Fourier coefficients, which we are supposed to find.
That means, we are in an overdetermined setting.
The second variation is the converse setting \mbox{$M>N$}, where we have to find more Fourier coefficients than we are given initial data.
Hence, this is the underdetermined case.
Analogously, the same relations can be considered for the inverse adjoint NFFT in~\eqref{eq:problem_infft*}.
There $M$ belongs to the given data whereas $N$ goes with the wanted solution.
Thus, the overdetermined case in now \mbox{$M>N$} while the problem is underdetermined for \mbox{$M<N$}.

This section is organized as follows.
Firstly, in Section~\ref{subsec:lagrange} the inversions are derived for the quadratic case \mbox{$M=N$}.
Secondly, in Section~\ref{subsubsec:ansatz1} we survey the underdetermined case of the inverse NFFT,
which corresponds to the overdetermined case of the adjoint.
Finally, in Section~\ref{subsubsec:ansatz2} the overdetermined case of the inverse NFFT is explained, which is related to the underdetermined case of the adjoint.

\subsection{The quadratic case\label{subsec:lagrange}}

For the quadratic case \mbox{$M=N$} we use an approach analogous to \cite{duro95, Selva18}, where an inversion is realized by means of Lagrange interpolation.
While the fast algorithms are obtained in~\cite{duro95} by means of FMM, our method from~\cite{KircheisBA17} employs the fast summation for acceleration, see~\cite{post02}.

The main idea is to use a relation between two evaluations of a trigonometric polynomial
\begin{equation*}
	f_j \coloneqq f(y_j) = \sum_{k = -\frac N2}^{\frac N2 -1} \hat f_k\, \e^{2\pi\i k y_j}, \quad j=1,\dots,N,
\end{equation*}
and
\begin{equation}
\label{eq:quadratic_eval_equispaced}
	g_l \coloneqq f(x_l) = \sum_{k = -\frac N2}^{\frac N2 -1} \hat f_k\, \e^{2\pi\i k x_l}, \quad l=1,\dots,N, 
\end{equation}
for different nodes~\mbox{$x_l, y_j \in \left[-\frac 12, \frac 12\right)$},\,\mbox{$l,j=1,\dots,N,$} and Fourier coefficients \mbox{$\hat f_k \in \C$},\, \mbox{$k = -\frac N2, \dots, \frac N2-1$}.
By defining the coefficients
\begin{equation}
\label{eq:coeff_ab}
	a_l = \prod_{n=1}^{N} \sin(\pi(x_l-y_n))
	\quad\text{ and }\quad
	b_j = \prod_{n=1\atop{n\ne j}}^{N} \frac{1}{\sin(\pi(y_j-y_n))},
	\quad
	l,j = 1,\dots, N,
\end{equation}
we observe the relation 
\begin{equation}
\label{eq:lagrange_formula}
	g_l = a_l \sum_{j=1}^{N} f_j\, b_j \left(\frac{1}{\tan(\pi(x_l-y_j))}-\i\right), \quad l = 1,\dots, N,
\end{equation}
cf. \cite[Theorem~2.3]{duro95}.
Hence, for given nonequispaced nodes~$y_j$ the computation of an inverse NFFT can be realized by choosing additional points~$x_l$ and applying formula~\eqref{eq:lagrange_formula}.
If these nodes~$x_l$ are chosen equidistantly we can compute the Fourier coefficients~$\hat f_k$ by simply applying an FFT to the coefficients~$g_l$ in \eqref{eq:quadratic_eval_equispaced}.

\begin{Remark}
It must be pointed out that the considered approach is only applicable for disjoint sets of nodes.
If this condition is violated this would mean division by zero in case of the coefficients~$b_j$, cf. \eqref{eq:coeff_ab}.
However, this can easily be avoided since we are free to choose the equispaced points $x_l$ distinct from $y_j$.
\ex
\end{Remark}

We approximate the coefficients~$g_l$ in \eqref{eq:lagrange_formula} by using the fast summation, see~\cite{post02}.
Considering the computation scheme we see that it is possible to compute
\begin{equation}
\label{eq:coeff_g_tilde}
	\tilde g_l \coloneqq \sum_{j=1}^{N} f_j\, b_j \cot(\pi(x_l-y_j)), \quad l=1,\dots,N,
\end{equation}
by means of the fast summation.
Then the wanted coefficients~$g_l$ can be obtained by
\begin{equation}
\label{eq:coeff_g}
	g_l = a_l \left(\tilde g_l - \i \cdot \sum_{j=1}^{N} f_j\, b_j\right), \quad l=1,\dots,N,
\end{equation}
where we only have to compute an additional scalar product of two vectors, which requires only $\mathcal O(N)$ arithmetic operations.
Considering the kernel \mbox{$K(x) = \cot(\pi x)$} in detail it becomes apparent that this function has not only a singularity at \mbox{$x=0$} but also at the boundary \mbox{$x=\pm 1$}. 
Thus, we have to make an effort.
For detailed information about the computation see~\cite{KircheisBA17}.

The coefficients~$a_l$ and~$b_j$ can also be computed efficiently by the fast summation because of the observations
\begin{equation*}
\label{eq:logarithm_coeff_a}
	\tilde a_l \coloneqq \ln  |a_l| 
	= 
	\ln \left|\,\prod_{n=1}^{N} \sin(\pi(x_l-y_n))\right|
	= 
	\sum_{n=1}^{N} \,\ln \left|\sin(\pi(x_l-y_n))\right|
\end{equation*}
and
\begin{align*} 
\label{eq:logarithm_coeff_b}
	\tilde b_j \coloneqq \ln |b_j|
	& = 
	\ln \left|\,\prod_{n=1\atop{n\ne j}}^{N} \frac{1}{\sin(\pi(y_j-y_n))}\right|
	=
	-\sum_{n=1\atop{n\ne j}}^{N} \,\ln\left|\sin(\pi(y_j-y_n))\right|.
\end{align*}
Therefore, it is possible to use the kernel~\mbox{$K(x) = \ln(|\sin(\pi x)|)$} to compute the absolute values and perform a sign correction afterwards to receive the signed coefficients~$a_l$ and~$b_j$.
Having a closer look at the kernel it becomes apparent that this function is also one-periodic and shows singularities at the same positions as the cotangent does. 
Hence, the computation works analogously.

\begin{Remark}
\label{remark:stability_lagrange}
A straightforward implementation of \eqref{eq:lagrange_formula} can lead to overflow and underflow issues.
The closer our nodes are the more formula \eqref{eq:lagrange_formula} suffers from these issues. 
However, having a look at the coefficients $a_l$ and $b_j$ we recognize that the decrease and the increase are in the same scale.
Thus, for $s\in \R$ we realize a stabilization by considering
\mbox{$|a_l| \cdot |b_j|
= 
\e^{\ln |a_l|} \, \e^{\ln |b_j|}
= 
\e^{\tilde a_l} \, \e^{\tilde b_j}
= 
\e^{\tilde a_l+s} \, \e^{\tilde b_j-s},$}
i.\,e., instead of computing products of the coefficients we use the additional factor~$\e^s$ to overcome numerical difficulties.
To bring the coefficients as close together as possible we choose
\mbox{$s = \min \left\{\min_j \{\tilde d_j\}, \min_l\{\tilde c_l\}\right\}$}.
For details see also \cite[./matlab/infft1d]{nfft3}.
\ex
\end{Remark}

Thus, we obtain the following fast algorithm.

\begin{algorithm}{Fast inverse NFFT - quadratic case}
	\label{alg:inverse_nfft_lagrange_fast}
	For \mbox{$N\in 2\N$} let be given equispaced nodes~\mbox{$x_l\in\left[-\frac 12, \frac 12\right),\, l=1,\dots,N$}, nonequispaced nodes~\mbox{$y_j \in \left[-\frac 12, \frac 12\right)$} and \mbox{$f_j\in\C,\, j=1,\dots,N$}.
	\begin{enumerate}
		\item Compute $\tilde a_l=\ln|a_l|$ and $\tilde b_j=\ln|b_j|$ by means of the fast summation using the kernel function \mbox{$K(x) = \ln(|\sin(\pi x)|)$}.
		\hfill $\mathcal O(N\log N)$
		\item Determine the stabilization factor $s$ and set $a_l = \e^{\tilde a_l+s}$ and $b_j = \e^{\tilde b_j-s}$.
		\hfill $\mathcal O(N)$
		\item Perform a sign correction for~$a_l$ and $b_j$.
		\hfill $\mathcal O(N)$\\
		(If the nodes have to be sorted, we end up with $\mathcal O(N\log N)$)
		\item Compute $\tilde g_l$ analogously to $\tilde a_l$ by means of the fast summation with the kernel function~\mbox{$K(x) = \cot(\pi x)$}, cf. \eqref{eq:coeff_g_tilde}.
		\hfill $\mathcal O(N\log N)$
		\item Compute $g_l$ via \eqref{eq:coeff_g}.
		\hfill $\mathcal O(N)$ 
		\item Compute
		\begin{equation*}
		\check f_k = \frac 1N \sum_{l=1}^{N} g_l \,\e^{-2\pi\i kx_l}, \quad k = -\tfrac N2, \dots, \tfrac N2-1,
		\end{equation*}
		by means of an FFT.
		\hfill $\mathcal O(N\log N)$	
	\end{enumerate}
	\vspace{2pt} 
	Output: $\check f_k \approx \hat f_k$
	\vspace{2pt} \\
	Complexity: $\mathcal O(N\log N)$
\end{algorithm}

\begin{Remark}
This algorithm is part of the software package NFFT 3.4.1, see \cite[./matlab/infft1d]{nfft3}.
\ex
\end{Remark}

Now we have a look at some numerical examples.

\begin{Example}
\label{ex:trig_poly_lagrange}
We choose arbitrary Fourier coefficients~\mbox{$\hat f_k \in [1,100]$} and compute the evaluations of the related trigonometric polynomial~\eqref{eq:trig_polynomial}.
Out of these we want to retrieve the given $\hat f_k$.
As mentioned in \cite{GeSo14, GeSo16, AuTr16} we examine so-called jittered equispaced nodes
\begin{equation}
\label{eq:jittered_nodes}
	y_j = -\frac 12 + \frac{j-1}{N} + \frac 1{4N}\,\theta, \quad j = 1, \dots, N, \ \text{with}\ \theta \sim U(0,1),
\end{equation}
where \mbox{$U(0,1)$} denotes the uniform distribution on the interval $(0,1)$ and the factor $\frac 14$ is our choice for the arbitrary perturbation parameter which has to be less than $\frac 12$ in order to not let the nodes switch position.
For a detailed study of this sampling pattern we refer to \cite[Section 4]{AuTr16}.
We consider the absolute and relative errors per node
\begin{equation}
\label{eq:errors_per_node_lagrange}
	\frac{e_r^{\mathrm{abs}}}{N} = \frac 1N \|\b{\hat f}-\b{\check f}\|_{r}
	\quad\text{ and }\quad
	\frac{e_r^{\mathrm{rel}}}{N} = \frac{\|\b{\hat f}-\b{\check f}\|_{r}}{N\,\|\b{\hat f}\|_{r}}
\end{equation}
for $r\in\{2,\infty\}$, where $\b{\check f}$ is the outcome of Algorithm~\ref{alg:inverse_nfft_lagrange_fast}.

As a first experiment we use \mbox{$N=2^c$} with \mbox{$c=1,\dots,14$}, and for the parameters needed in the fast summation we use the standard values, see~\cite{nfft3}.
In a second experiment we fix \mbox{$N=1024$} and increase some of the standard values, namely the cut-off parameter $m$ and the degree of smoothness $p$ shall be chosen uniformly \mbox{$m=p=c$} with \mbox{$c=4,\dots,12$}.

The corresponding results are depicted in Figure~\ref{fig:errors_lagrange}.
Having a look at the errors per node for growing $N$, see (a), we observe that the errors are worse if we consider very small sizes of $N$.
Otherwise, we recognize that these errors remain stable for large sizes of $N$.
In (b) we can see that for fixed $N$ a higher accuracy can be achieved by tuning the parameters of the fast summation.
\begin{figure}[!h]
	\centering
	\captionsetup[subfigure]{justification=centering}
	\begin{subfigure}{0.4\textwidth}
		\includegraphics[width=0.9\textwidth]{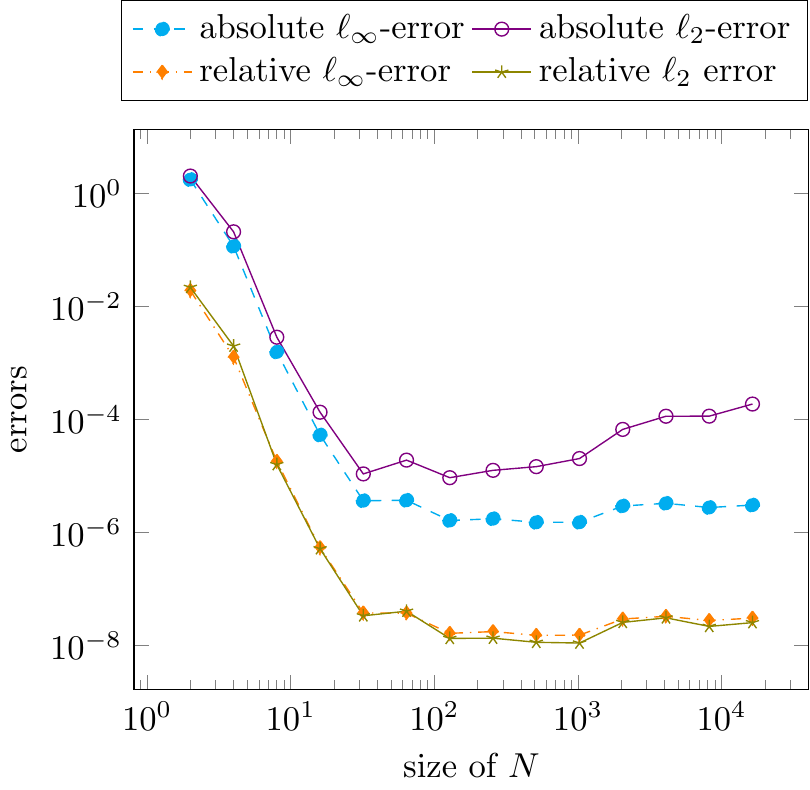}
		\caption{$N=2^c,\, c=1,\dots,14,$ \\ and $m=p=4$.}
	\end{subfigure}
	\begin{subfigure}{0.4\textwidth}
		\includegraphics[width=0.9\textwidth]{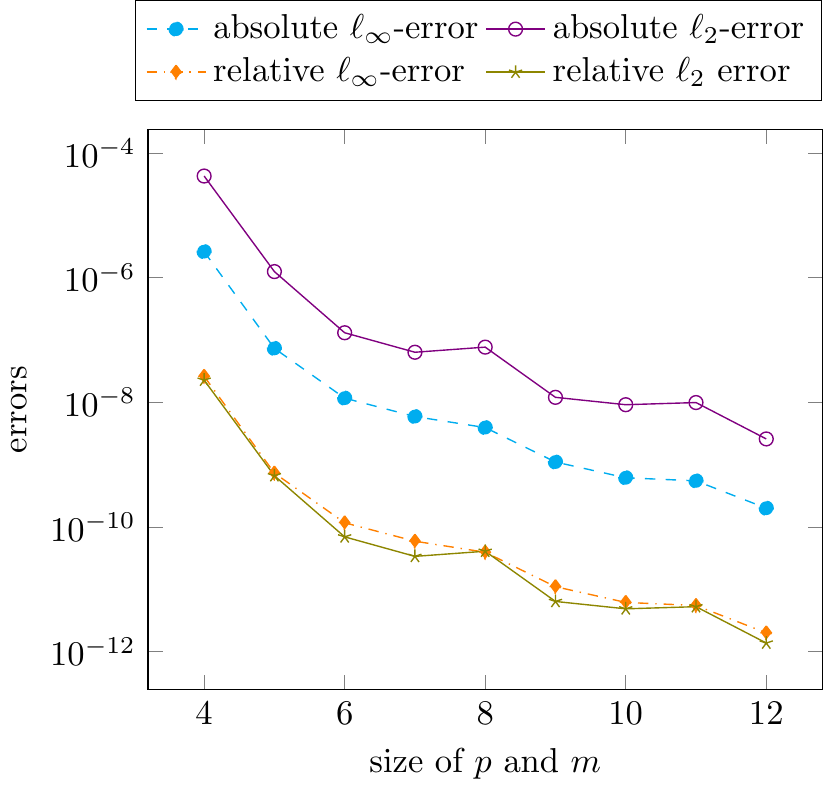}
		\caption{$N=1024$ and \\ $m=p=c,\, c=4,\dots,12$.}
	\end{subfigure}
	\caption{
		Comparison of the errors per node~\eqref{eq:errors_per_node_lagrange} of reconstructed Fourier coefficients generated by Algorithm~\ref{alg:inverse_nfft_lagrange_fast} for jittered equispaced nodes.
		\label{fig:errors_lagrange}}
\end{figure}
\ex
\end{Example}

\begin{Remark}
We have a look at the condition number of the nonequispaced Fourier matrix~$\b A$.
Figure~\ref{fig:condition_number} displays $\cond_2(\b A)=\|\b A\|_2 \|\b A^{-1}\|_2$ for different kinds of nodes for increasing \mbox{$N=2^c,\,c=1,\dots,10$}.
There we see that the distribution of the nonequispaced nodes is of great importance.
For jittered equispaced nodes, cf. \eqref{eq:jittered_nodes}, the condition number is nearly 1 for all sizes of $N$, so this problem is well-conditioned.
However, for Chebyshev nodes
\begin{equation}
\label{eq:tschebyscheff_nodes}
	y_j = \frac 12 \cos \left(\frac{2(N- j)+1}{2N}\ \pi \right), \quad j = 1, \dots, N,
\end{equation}
or logarithmically spaced nodes  
\begin{equation*}
\label{eq:log_nodes}
	y_j = \left(\frac 65 \right)^{j-N}, \quad j = 1, \dots, N,
\end{equation*}
it is easy to see that the condition number rises rapidly.
These last mentioned problems are simply ill-conditioned and we cannot assume a good approximation by Algorithm~\ref{alg:inverse_nfft_lagrange_fast}.
For a detailed investigation of the condition number for rectangular nonequispaced Fourier matrices see \cite{KuNa18} and the references therein and also \cite{BaGr03, BoPo06}.
\begin{figure}[!h]
	\centering
	\includegraphics[width=0.38\textwidth]{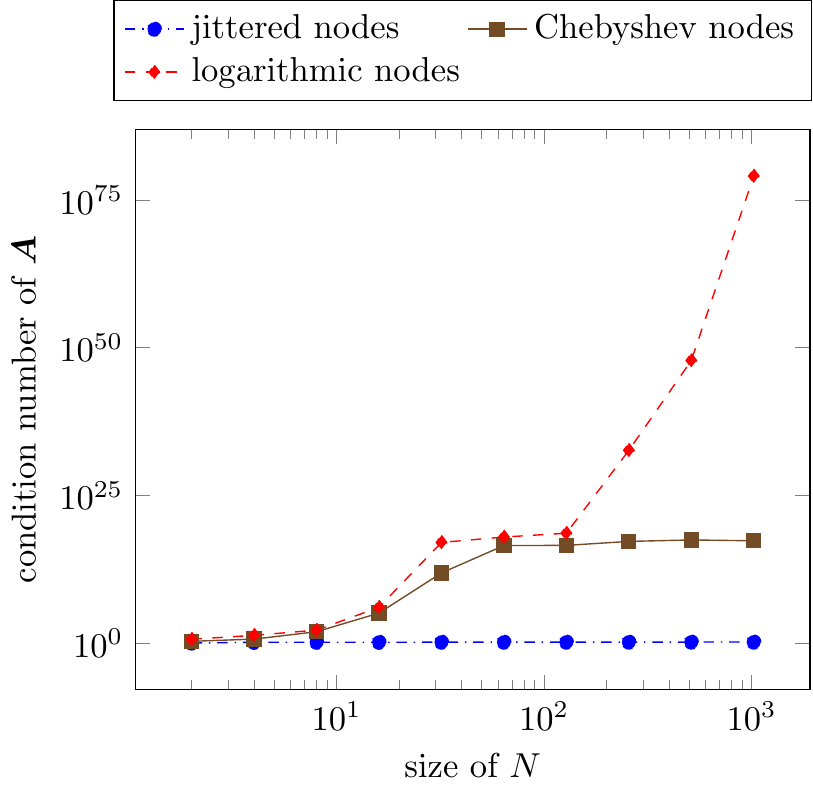}
	\caption{
		Comparison of condition numbers of the nonequispaced Fourier matrix $\b A$ computed using cond from MATLAB for different kinds of nodes for \mbox{$N=2^c,\, c=1,\dots,10$}.
		\label{fig:condition_number}}
\end{figure}
\ex
\end{Remark}

\begin{Remark}
We obtain an inverse adjoint NFFT by simply considering the adjoint of Algorithm~\ref{alg:inverse_nfft_lagrange_fast}, i.\,e., for $\b v \coloneqq (v_l)_{l=1}^N$ being the inverse Fourier transform of~$\b h$ we apply the formula
\begin{equation}
\label{eq:lagrange_formula_adjoint}
	f_j = b_j \sum_{l=1}^{N} a_l\, v_l \left(\frac{1}{\tan(\pi(x_l-y_j))}+\i\right), \quad j = 1,\dots, N.
\end{equation}
This relation can easily be seen by using the matrix representation of Algorithm~\ref{alg:inverse_nfft_lagrange_fast}.
\ex
\end{Remark}

\subsection{The rectangular case\label{subsec:rect}}

For the general case \mbox{$M\neq N$} we follow a different approach.
To clarify the idea we firstly have a look at equispaced nodes 
\begin{equation*}
	x_j = \tfrac jN \in \left[-\tfrac 12,\tfrac 12 \right),\, j = -\tfrac N2,\dots,\tfrac N2-1.
\end{equation*} 
Thereby, we obtain  
\begin{equation*}
	\b A = \left(\e^{2\pi\i k \frac jN}\right)_{j=-\frac N2,\, k=-\frac M2}^{\frac N2 -1,\; \frac M2 -1}
	\quad \text{ and } \quad
	\b A^* = \left(\e^{-2\pi\i k \frac jN}\right)_{k=-\frac M2,\,j=-\frac N2}^{\frac M2 -1,\;\frac N2 -1}.
\end{equation*}
Considering products of these two matrices it becomes apparent that
\mbox{$\b A^* \b A = N \b I_M$} for \mbox{$M \leq N$}
as well as
\mbox{$\b A \b A^* = M \b I_N$} for \mbox{$M\geq N$} with \mbox{$N \mid M$}.
This is to say, in these special cases we are given an inversion of the NFFT by composition of the Fourier matrices.
Hence, we seek to use this result to find a good approximation of the inversion in the general case.
This will be done by modification of the matrix~$\b B$ so that we receive an approximation of the form 
\mbox{$\b A \b D^* \b F^* \b B^* \approx M \b I_N$} similar to the equispaced case.
For this purpose, the entries of the matrix~$\b B$ should be calculated such that its sparse structure with at most \mbox{$(2m+1)$} entries per row and consequently the arithmetic complexity of the algorithms is preserved.
A matrix~$\b B$ satisfying this property we call \mbox{\textbf{$\textbf{(2m+1)}$-sparse}}.

It is to be noticed that the fact of underdetermination and overdetermination is not of great importance when deducing the methods for the inversion.
Even if it is a necessary condition for the exact inversion for equispaced nodes, the algorithms in the nonequispaced setting can always be used in both cases.
However, we will see later on that each algorithm works best in one of these cases and therefore they are already introduced for this special case.
Having this in mind we give an outline of how to handle problems~\eqref{eq:problem_infft} and \eqref{eq:problem_infft*}.
\begin{enumerate}
	\item[(1)] 
	To solve~\eqref{eq:problem_infft} our aim is to compute a sparse matrix~$\b B^*$ from given nodes~$x_j$ such that by application of an adjoint NFFT we obtain a fast inverse NFFT.	
	
	Suppose we are given the approximation \mbox{$\b A \b D^* \b F^* \b B^* \approx M \b I_N$}. 
	Then it also holds that 
	\begin{equation}
	\label{eq:approx_ansatz1}
		\frac 1M\ \b A \b D^* \b F^* \b B^* \b f \approx \b f \quad \forall \b f \in \C^N.
	\end{equation}
	If we now set
	\begin{equation*}
		\b{\check f} := \frac 1M\, \b D^* \b F^* \b B^* \b f,
	\end{equation*} 
	we can rewrite approximation~\eqref{eq:approx_ansatz1} as \mbox{$\b A \b{\check f} \approx \b f$}.
	Since we already know that \mbox{$\b A \b{\hat f} = \b f$} this means \mbox{$\b{\check f} \approx \b{\hat f}$}, which could be interpreted as a reconstruction of the Fourier coefficients~$\b{\hat f}$.
	To achieve a good approximation we want $\b{\check f}$ to be as close as possible by $\b{\hat f}$. 
	This can be accomplished by optimizing \mbox{$\b A \b{\check f} \approx \b f$}, i.\,e., we aim to solve the optimization problem
	\begin{equation*}
		\underset{\b{\check f}\in\C^M}{\text{Minimize }} \ 
		\|\b A \b{\check f} - \b f\|_2.
	\end{equation*}
	Using the definition of $\b{\check f}$ this norm can be estimated above by
	\begin{equation*}
		\|M \b A \b{\check f} - M \b f\|_2
		=
		\|\b A \b D^* \b F^* \b B^* \b f - M \b f\|_2 
		\leq
		\|\b A \b D^* \b F^* \b B^* - M \b I_N\|_{\mathrm F}\, \|\b f\|_2,
	\end{equation*}
	where the Frobenius norm is denoted by \mbox{$\|\cdot\|_{\mathrm F}$}.
	Because $\b f$ is given, this expression can be minimized by solving 
	\begin{equation}
	\label{eq:min_norm_ansatz1}
		\underset{\b B \in \R^{N\times M_\sigma} \colon \b B\, (2m+1)\text{-sparse }}{\text{Minimize }} \ 
		\|\b A \b D^* \b F^* \b B^* - M \b I_N\|_{\mathrm F}^2.
	\end{equation}
	\item[(2)]
	To solve \eqref{eq:problem_infft*} we aim to compute a sparse matrix $\b B$ from given nodes $x_j$ such that by application of an NFFT we obtain a fast inverse adjoint NFFT.
	
	Again we suppose \mbox{$\b A \b D^* \b F^* \b B^* \approx M \b I_N$}, which is equivalent to its adjoint
	\begin{equation*}
		\b B \b F \b D \b A^* \approx M\, \b I_N \quad \text{and\ } \quad
		\frac 1M\, \b B \b F \b D\, (\b A^*  \b f) \approx \b f \quad \forall \b f \in \C^N,
	\end{equation*}
	respectively. 
	Because we know \mbox{$\b h = \b A^* \b f$},
	this could be interpreted as a reconstruction of the coefficients~$\b f$.
	To achieve a good approximation we solve the optimization problem 
	\begin{equation*}
		\underset{\b B \in \R^{N\times M_\sigma} \colon \b B\, (2m+1)\text{-sparse }}{\text{Minimize }} \ 
		\|\b B \b F \b D \b h - M \b f\|_2,
	\end{equation*}
	where the norm could be estimated as follows.
	\begin{equation*}
		\|\b B \b F \b D \b h - M \b f\|_2
		=
		\|\b B \b F \b D \b A^* \b f - M \b f\|_2 
		\leq
		\|\b B \b F \b D \b A^* - M \b I_N\|_{\mathrm F}\, \|\b f\|_2.
	\end{equation*}
	Hence, we end up with the optimization problem
	\begin{equation*}
		\underset{\b B \in \R^{N\times M_\sigma} \colon \b B\, (2m+1)\text{-sparse }}{\text{Minimize }} \ 
		\|\b B \b F \b D \b A^* - M \b I_N\|_{\mathrm F}^2.
	\end{equation*}
\end{enumerate}
So, all in all, with the chosen approach we are able to generate an inverse NFFT as well as an inverse adjoint NFFT by modifying the matrices~$\b B^*$ and~$\b B$, respectively, and applying an adjoint NFFT or an NFFT with these modified matrices.

\begin{Remark}
We investigate below if the reconstruction error can be reduced by appropriate choice of the entries of the matrix~$\b B$. 
Already in~\cite{St01} the minimization of the Frobenius norm 
\mbox{$\|\b A - \b{BFD}\|_{\textrm F}$}
was analyzed regarding a sparse matrix~$\b B$ to achieve a minimum error for the NFFT.
In contrast, we study the minimization of  
\mbox{$\|\b A \b D^* \b F^* \b B^* - M \b I_N\|_{\textrm F}^2$}
to achieve a minimum error for the inverse NFFT
as well as the minimization of 
\mbox{$\|\b B \b F \b D \b A^* - M \b I_N\|_{\textrm F}^2$}
to achieve a minimum error for the inverse adjoint NFFT.
\ex
\end{Remark}

\subsubsection{Inverse NFFT -- underdetermined case\label{subsubsec:ansatz1}}

We start deducing our inversion as outlined in general.
However, in the numerical experiments in Examples~\ref{ex:matrixnorm_ansatz1} and~\ref{ex:trig_poly_ansatz1} we will see that this method is especially beneficial for the underdetermined setting and hence it is already attributed to this case.

As mentioned before we aim to find a solution for \eqref{eq:problem_infft} by solving~\eqref{eq:min_norm_ansatz1}.
Therefore, we consider the matrix 
\mbox{$\b A \b D^* \b F^* \b B^*$}
for given nodes \mbox{$x_j\in\T$},\, $j=1,\dots,N$. 
Apparently, we have
\begin{equation}
\label{eq:matrix_K}
	\b A \b D^* \b F^*
	=
	\left[
	\frac 1{M_\sigma} \sum_{k=-\frac M2}^{\frac M2 -1} 
	\frac{1}{\hat{w}(-k)}\, \e^{2\pi\i k \left(x_j-\frac{l}{M_\sigma}\right)}
	\right]_{j=1,\,l=-\frac{M_\sigma}{2}}^{N, \frac{M_\sigma}{2}-1}.
\end{equation}
By defining the ``inverse window function''
\begin{equation}
\label{eq:kernel}
	K(x) = \frac{1}{M_\sigma} \sum_{k=-\frac M2}^{\frac M2-1} \frac{1}{\hat w(-k)} \,\e^{2\pi\i k x}
\end{equation}
we receive
\begin{equation*}
	\b K 
	\coloneqq 
	\b A \b D^* \b F^*
	=
	\left( K \hspace{-2pt} \left(x_h-\tfrac{l}{M_\sigma}\right) \right)
	_{h=1,\, l=-\frac{M_\sigma}{2}}
	^{N,\ \frac{M_\sigma}{2}-1}.
\end{equation*}
Having a look at the matrix~$\b B^*$ it becomes apparent that there are only a few nonzero entries.
Thus, we study the window~$\tilde w_m$ for further simplification.
For $w_m$ we have 
\mbox{$\text{supp}(w_m) = \left[-\tfrac{m}{M_\sigma}, \tfrac{m}{M_\sigma}\right]$},
i.\,e., for the 1-periodic version \mbox{$\tilde w_m(x):=\sum_{z\in \mathbb Z} w_m(x+z)$} it holds
\begin{align*}
	\tilde w_m \hspace{-2.5pt} \left(x_j-\tfrac{l}{M_\sigma}\right) \neq 0 
	& \iff \exists\, z \in \mathbb Z:\ -m \leq M_{\sigma} x_j - l + M_{\sigma} z \leq m.
\end{align*}
By defining the set
\begin{equation}
\label{eq:set_xj}
	I_{M_{\sigma},m}(x_j) 
	\coloneqq 
	\left\{l \in \left\{-\tfrac{M_\sigma}{2}, \dots, \tfrac{M_\sigma}{2}-1 \right\}: 
	\exists\, z\in \mathbb Z\ \text{with} -m \leq M_{\sigma} x_j - l + M_{\sigma} z \leq m \right\}
\end{equation}
we can therefore write
\begin{align*}
	\left(\b A \b D^* \b F^* \b B^*\right)_{h,j} \notag 
	&= 
	\sum_{k=-\frac M2}^{\frac M2 -1} 
	\e^{2\pi\i kx_h} \hspace{-3pt} 
	\left( 	
	\sum_{l \in I_{M_{\sigma},m}(x_j)} 
	\frac{1}{M_\sigma \hat{w}(-k)}\, 
	\e^{-2\pi\i k\frac{l}{M_\sigma}}\ 
	\tilde w_m \hspace{-2.5pt} \left(x_j - \tfrac l{M_{\sigma}}\right)
	\hspace{-3pt}\right).
\end{align*}
Hence, our considered norm can be written as
\begin{align}
\label{eq:norm_frobenius}
	\|\b A \b D^* &\b F^* \b B^* - M \b I_N\|_{\textrm F}^2 = \\
	&=
	\sum_{h=1}^{N} \sum_{j=1}^{N} 
	\left|\hspace{-3pt}
	\ \sum_{k=-\frac M2}^{\frac M2 -1} 
	\e^{2\pi\i kx_h} \hspace{-3pt}
	\left( 	
	\sum_{l \in I_{M_{\sigma},m}(x_j)} 
	\frac{1}{M_\sigma \hat{w}(-k)}\, 
	\e^{-2\pi\i k\frac{l}{M_\sigma}}\, 
	\tilde w_m \hspace{-2.5pt} \left(x_j - \tfrac l{M_{\sigma}}\right)
	\hspace{-3pt}\right)
	- M \delta_{hj}
	\right|^2. \notag  
\end{align}
Based on the definition of the Frobenius norm of a matrix~\mbox{$\b A \in \R^{k\times n}$} 
we obtain for $\b a_j$ being columns of~\mbox{$\b A \in \R^{k\times n}$} that
\begin{equation*}
	\|\b A\|_F^2 = \sum_{i=1}^{k} \sum_{j=1}^{n} |a_{ij}|^2 = \sum_{j=1}^{n} \|\b a_j\|_2^2.
\end{equation*}
This yields that \eqref{eq:norm_frobenius} can be rewritten by means of
\begin{equation*}
	\b T_j = \left(\e^{-2\pi\i k\frac{l}{M_\sigma}}\right)_{k=-\frac M2,\, l \in I_{M_{\sigma},m}(x_j)}^{\frac M2-1}, 
	\quad
	\b b_j = \bigg(\tilde w_m \hspace{-2.5pt} \left(x_j - \tfrac l{M_{\sigma}}\right)\bigg)
	_{l \in I_{M_{\sigma},m}(x_j)}
\end{equation*}
and \mbox{$\b e_j = (\delta_{hj})_{h=1}^N$} as
\begin{equation}
\label{eq:min_decomp_ansatz1}
	\|\b A \b D^* \b F^* \b B^* - M \b I_N\|_{\textrm F}^2 
	= \sum_{j=1}^N 
	\|\b A \b D^* \b T_j \b b_j - M \b e_j\|_2^2.
\end{equation}
Therefore, the considered norm in \eqref{eq:min_norm_ansatz1} is minimal if and only if 
\mbox{$\|\b A \b D^* \b T_ j \b b_j - M \b e_j\|_2^2$}
is minimal for all \mbox{$j=1,\dots, N$}.
Hence, we obtain the optimization problems 
\begin{equation}
\label{eq:opt_ansatz1}
	\underset{\tilde{\b b}_j \in \mathbb R^{2m+1}}{\text{Minimize }}\ 
	\|\b A \b D^* \b T_j \tilde{\b b}_j - M \b e_j\|_2^2, 
	\quad j=1,\dots,N,
\end{equation}
since the columns of the matrix~$\b B^*$ contain at most \mbox{$(2m+1)$} nonzeros.
Thus, if 
\begin{equation}
\label{eq:matrix_Kj}
	\hspace{-0.2cm} \b K_j
	\coloneqq
	\b A \b D^* \b T_j 
	=
	\left[
	\frac 1{M_\sigma} \sum_{k=-\frac M2}^{\frac M2 -1} 
	\frac{1}{\hat{w}(-k)}\, \e^{2\pi\i k \left(x_h-\frac{l}{M_\sigma}\right)}
	\right]_{h=1,\,l \in I_{M_{\sigma},m}(x_j)}^{N} \hspace{-35pt}\in \mathbb C^{N\times (2m+1)}
\end{equation} 
has full rank the solution of problem~\eqref{eq:opt_ansatz1} is given by
\begin{equation}
\label{eq:solution_ansatz1}
	\tilde{\b b}_j 
	= M
	\left(\b K_j^* \b K_j\right)^{-1} \b K_j^*  \b e_j, 
	\quad j = 1, \dots, N.
\end{equation}
For generating the modified matrix $\b B_{\mathrm{opt}}^*$ it must be pointed out that the vectors~$\tilde{\b b}_j$ only contain the \mbox{$(2m+1)$} nonzeros of the columns of $\b B_{\mathrm{opt}}^*$.
Hence, attention must be paid to the periodicity which can also be seen in the structure of the matrix~$\b B^*$.

\begin{Remark}
Whether the matrix~$\b K_j$ has full rank only depends on the matrix~$\b A$.
The conditions when this matrix has full rank can e.\,g. be found in \cite{Groechenig92} and~\cite{kupo04}.
\ex
\end{Remark}

Next we develop a method for the fast computation of $\tilde{\b b}_j$.
It is already known from Section~\ref{sec:nfft} that sums of the form~\eqref{eq:nfft}
can be computed in \mbox{$\mathcal O(M\log M + N)$} arithmetic operations for given nodes 
\mbox{$x_j \in \left[-\frac 12, \frac 12\right)$}, \mbox{$j=1,\dots,N,$} 
and coefficients 
\mbox{$\hat f_k \in \C,\, k=-\frac M2,\dots,\frac M2-1$}.
If we have a look at the matrix~$\b K_j$, cf.~\eqref{eq:matrix_Kj}, it becomes apparent that we can compute its entries by means of the NFFT with coefficients
\begin{equation}
\label{eq:fourier_coefficients_fast_ansatz1}
	\hat f_k = \frac{1}{M_\sigma \hat w(-k)},
	\quad
	k=-\tfrac M2,\dots,\tfrac M2-1,
\end{equation}
and nodes
\mbox{$y_{h,l} \coloneqq x_h -\tfrac{l}{M_\sigma},\, h=1,\dots,N, l \in I_{M_{\sigma},m}(x_j)$},
which are at most \mbox{$N(2m+1)$} many.
If we put the columns of $\b K_j$ one below the other into a vector, we are able to compute these entries only using one NFFT of length \mbox{$N(2m+1)$}.
In so doing, we have to reshape the obtained vector into a matrix afterwards.

Another point to mention is that the coefficients~$\hat f_k$ are the same for the computation of all matrices
$\b K_j$,$\, j=1,\dots,N$.
This is to say, we can precompute step 1 and step 2 of the NFFT since there only information about the Fourier coefficients~$\hat f_k$ is needed, cf.~\cite{postta01}.
Merely for the last step of the NFFT we need the current nodes and therefore this step has to be performed separately for every \mbox{$j=1,\dots,N$}.
Thus, we receive the following algorithm.

\begin{algorithm}{Fast optimization of the matrix $\b B^*$}
	\label{alg:opt_fast_ansatz1}
	For $N \in \N$ let $x_j \in \left[-\frac 12, \frac 12\right),\, j=1,\dots,N,$ be given nodes as well as $M \in 2\N$, $\sigma \geq 1$ and $M_\sigma = \sigma M$.
	Furthermore, we are given the oversampling factor $\sigma_2\geq 1$ and the cut-off parameter $m_2$ for an NFFT.
	\begin{enumerate}	
		\item Compute $\b g = \b F \b D \b{\hat f}$, cf.~\eqref{eq:matrix_D} and \eqref{eq:matrix_F}, with $\hat f_k$ in \eqref{eq:fourier_coefficients_fast_ansatz1}.
		\hfill $\mathcal O(M\log M)$
		\item
		For $j = 1, \dots, N$: 
		\begin{itemize}
			\item[] \hspace{-0.6cm} Determine the set $I_{M_{\sigma},m}(x_j)$, cf. \eqref{eq:set_xj}.
			\hfill $\mathcal O(1)$
			\item[] \hspace{-0.6cm} Perform $\b B \b g$, cf.~\eqref{eq:matrix_B}, for the vector of nodes
			\begin{equation*}
			\b y \coloneqq \left( \b y_1^T, \dots, \b y_s^T \right)^T
			\end{equation*} 
			\hspace{-0.6cm} for $\b y_n$ being the columns of the matrix 
			$\b Y \coloneqq (y_{h,l})_{h=1,\,l \in I_{M_{\sigma},m}(x_j)}^{N}$.
			\hfill $\mathcal O(N)$
			\item[] \hspace{-0.6cm} Reshape the obtained vector into the matrix $\b K_j \in \C^{N \times (2m+1)}$.
			\hfill $\mathcal O(N)$
			\item[] \hspace{-0.6cm} Solve the normal equations for \(\b K_j\), 
			cf. \eqref{eq:solution_ansatz1}. \hfill $\mathcal O(N)$
		\end{itemize}
		\item Compose $\b B_{\mathrm{opt}}^*$ column-wise of the vectors $\tilde{\b b}_j$ observing the periodicity.
		\hfill $\mathcal O(N)$
	\end{enumerate}
	\vspace{2pt} 
	Output: optimized matrix $\b B_{\mathrm{opt}}^*$ 
	\vspace{2pt}\\
	Complexity: $\mathcal O(N^2 + M\log{M})$
\end{algorithm}

\begin{Remark}
\label{remark:Dirichlet_ansatz1}
It is possible to simplify Algorithm~\ref{alg:opt_fast_ansatz1} by replacing the inverse window function in~\eqref{eq:kernel} by the Dirichlet kernel
\begin{equation}
\label{eq:kernel_Dirichlet}
	D_{\frac M2-1}(x) 
	=
	\sum_{k=-\frac M2+1}^{\frac M2-1} \e^{2\pi\i kx}
	=
	\frac{\sin((M-1)\pi x)}{\sin(\pi x)}.
\end{equation}
%
%
Hence, the entries of $\b K_j$ in \eqref{eq:matrix_Kj} can explicitly be stated by means of \eqref{eq:kernel_Dirichlet} as
\begin{equation*}
	\b K_j
	= 
	\left[
	\frac 1{M_\sigma} \,
	D_{\frac M2-1}\left(x_h-\tfrac{l}{M_\sigma}\right)
	\right]_{h=1,\,l \in I_{M_{\sigma},m}(x_j)}^{N} 
\end{equation*}
and thereby the term \mbox{$M\log M$} in the computational costs of Algorithm~\ref{alg:opt_fast_ansatz1} is eliminated.
Thus, we end up with an arithmetic complexity of \mbox{$\mathcal O(N^2)$}.
\hfill \ex			
\end{Remark}

\begin{Remark}
This algorithm using a window function or the Dirichlet kernel is part of the software package NFFT 3.5.1, see \cite[./matlab/infft1d]{nfft3}.
\ex
\end{Remark}

Now we have a look at some numerical examples.

\begin{Example}
\label{ex:matrixnorm_ansatz1}
Firstly, we verify that the optimization was successful.
To this end, we compare the norms
\begin{equation}
\label{eq:matrixnorm_opt_ansatz1}
	\| \b A \b D^* \b F^* \b B^* - M \b I_N \|_{\textrm F}
	\quad\text{ and }\quad
	\| \b A \b D^* \b F^* \b B_{\mathrm{opt}}^* - M \b I_N \|_{\textrm F},
\end{equation}
where $\b B^*$ denotes the original matrix from the adjoint NFFT and $\b B_{\mathrm{opt}}^*$ the optimized matrix generated by Algorithm~\ref{alg:opt_fast_ansatz1}.
Even though our method is attributed to the underdetermined setting, this is not a restriction.
Hence, we also test for the overdetermined setting.
\begin{itemize}
	\item[(i)] 
	Firstly, we choose \mbox{$N=128$} jittered equispaced nodes, cf. \eqref{eq:jittered_nodes}, \mbox{$M=2^c$} with \mbox{$c=4,\dots,12,$} and for the NFFT in Algorithm~\ref{alg:opt_fast_ansatz1} we choose the Kaiser-Bessel window, \mbox{$\sigma_2=2.0$} and \mbox{$m_2=2m$} to receive high accuracy.
	Figure~\ref{fig:matrixnorm_Bspline_jittered_ansatz1} depicts the comparison of the norms \eqref{eq:matrixnorm_opt_ansatz1} for different values of $m$ and~$\sigma$ for $\b B_{\mathrm{opt}}^*$ generated using B-Splines as well as the Dirichlet kernel as mentioned in Remark~\ref{remark:Dirichlet_ansatz1}.
	It can be seen that the optimization was really successful for large values of~$M$ compared to~$N$ whereas the minimization does not work in the overdetermined setting \mbox{$M<N$}.
	But in fact, this is not surprising because then we try to approximate the identity by a low rank matrix since \mbox{$\b A \b D^* \b F^* \b B^* \in \C^{N\times N}$} has at most rank~$M$.
	Therefore, Algorithm~\ref{alg:opt_fast_ansatz1} is specially attributed to the underdetermined case.
	We recognize that the optimization is worsened against expectation by using a higher oversampling factor~$\sigma$
	whereas increasing the cut-off~$m$ expectedly leads to better results.
	
	Figure~\ref{fig:runtime_Dirichlet_ansatz1} displays the run-times of Algorithm~\ref{alg:opt_fast_ansatz1} comparing B-Spline and Dirichlet kernel for \mbox{$m=2$} and \mbox{$\sigma=1.0$}.
	It is obvious that the usage of the Dirichlet kernel considerably reduces the run-time. 
	Since the results are the same for other parameters and window functions additional tests are omitted.
	\item[(ii)] 
	Next we repeat the example for Chebyshev nodes, cf. \eqref{eq:tschebyscheff_nodes}.
	The corresponding outcomes for \mbox{B-Splines} can be found in Figure~\ref{fig:matrixnorm_Bspline_tsch_ansatz1}.
	There we see that the gap between $M$ and~$N$ has to be huge to achieve results similar to those for jittered equispaced nodes.
	\ex
\end{itemize}
\begin{figure}[p]
	\centering
	\begin{subfigure}{0.31\textwidth}
		\includegraphics[width=0.9\textwidth]{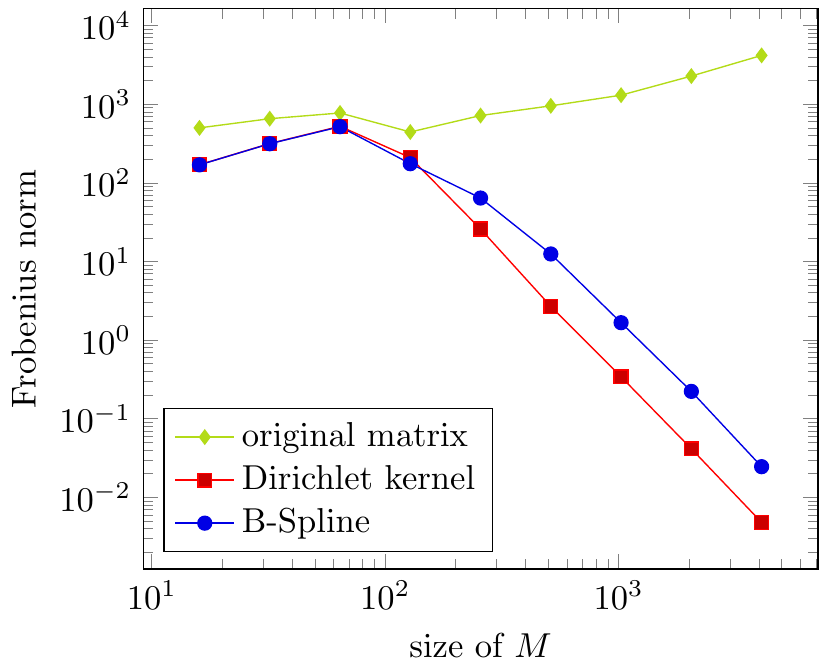}
		\caption{$m=2$ and $\sigma=1.0$}
	\end{subfigure}
	\begin{subfigure}{0.31\textwidth}
		\includegraphics[width=0.9\textwidth]{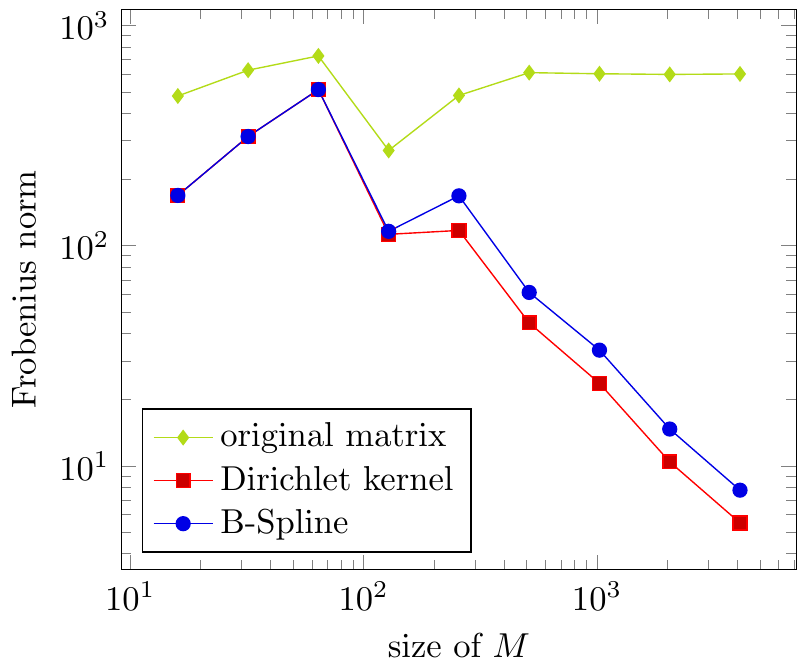}
		\caption{$m=2$ and $\sigma=2.0$}
	\end{subfigure}
	\begin{subfigure}{0.31\textwidth}
		\includegraphics[width=0.9\textwidth]{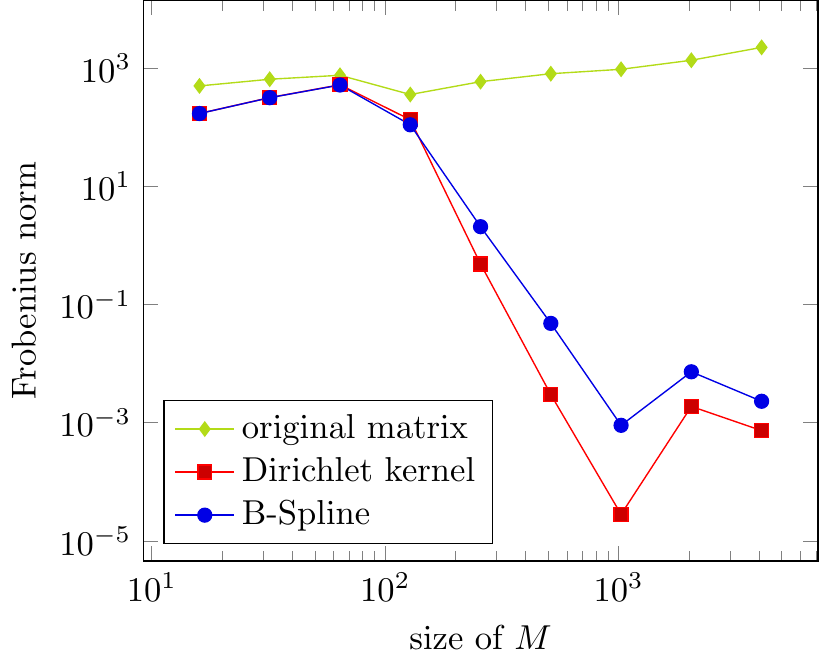}
		\caption{$m=4$ and $\sigma=1.0$}
	\end{subfigure}
	\caption{
		Comparison of Frobenius norms \eqref{eq:matrixnorm_opt_ansatz1} of the original matrix~$\b B^*$ and the optimized matrix~$\b B_{\mathrm{opt}}^*$ generated by Algorithm~\ref{alg:opt_fast_ansatz1} using \textbf{B-Spline} window functions as well as the Dirichlet kernel for \mbox{$N=128$} \textbf{jittered equispaced nodes} and \mbox{$M=2^c$} with \mbox{$c=4,\dots,12$}.
		\label{fig:matrixnorm_Bspline_jittered_ansatz1}}
\end{figure}
%
%
\begin{figure}[p]
	\centering
	\includegraphics[width=0.38\textwidth]{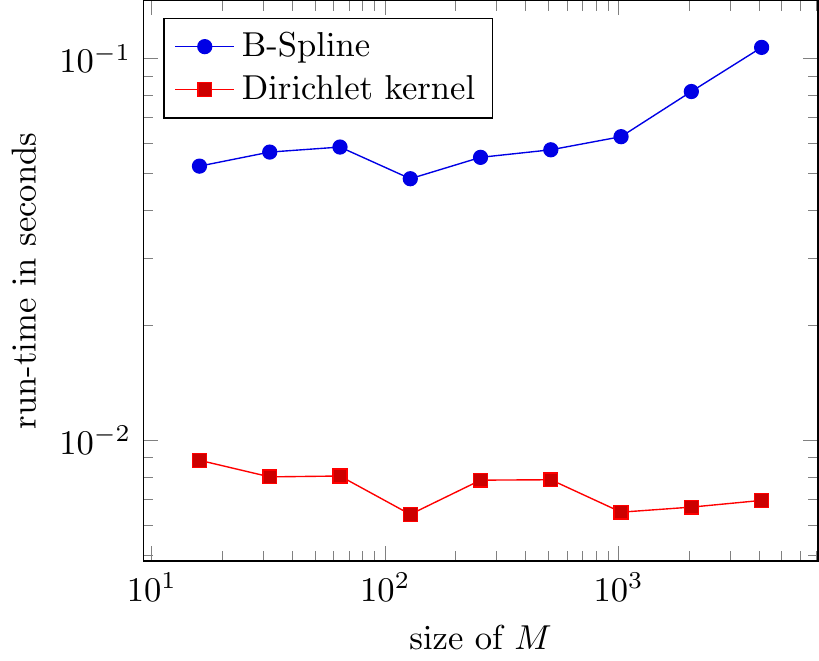}
	\caption{
		Comparison of run-times of Algorithm~\ref{alg:opt_fast_ansatz1} using the \mbox{\textbf{B-Spline}} as well as the Dirichlet kernel with \mbox{$m=2$} and \mbox{$\sigma=1.0$} for \mbox{$N=128$} \textbf{jittered equispaced nodes} and \mbox{$M=2^c$} with \mbox{$c=4,\dots,12$}.
		\label{fig:runtime_Dirichlet_ansatz1}}
\end{figure}
\begin{figure}[p]
	\centering
	\begin{subfigure}{0.31\textwidth}
		\includegraphics[width=0.9\textwidth]{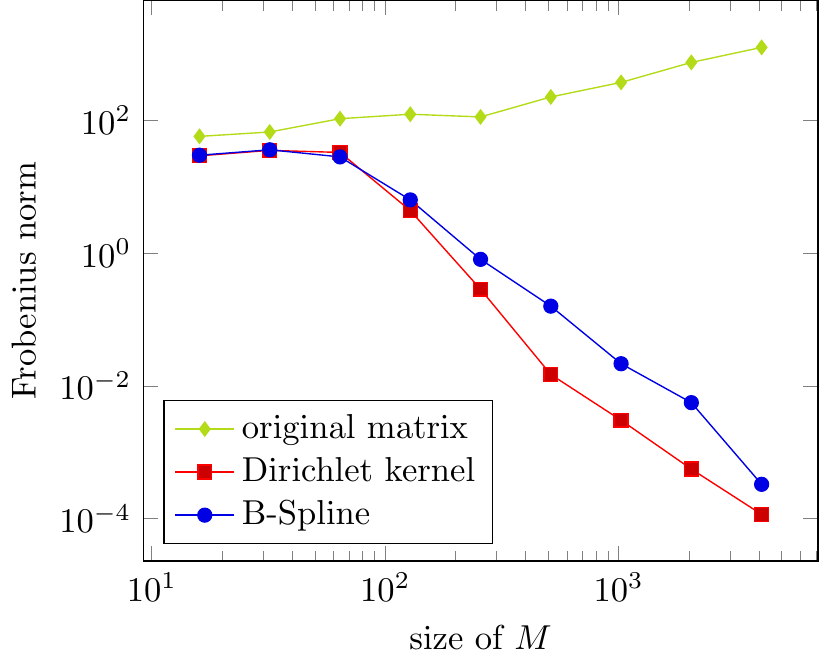}
		\caption{$N=16$}
	\end{subfigure}
	\begin{subfigure}{0.31\textwidth}
		\includegraphics[width=0.9\textwidth]{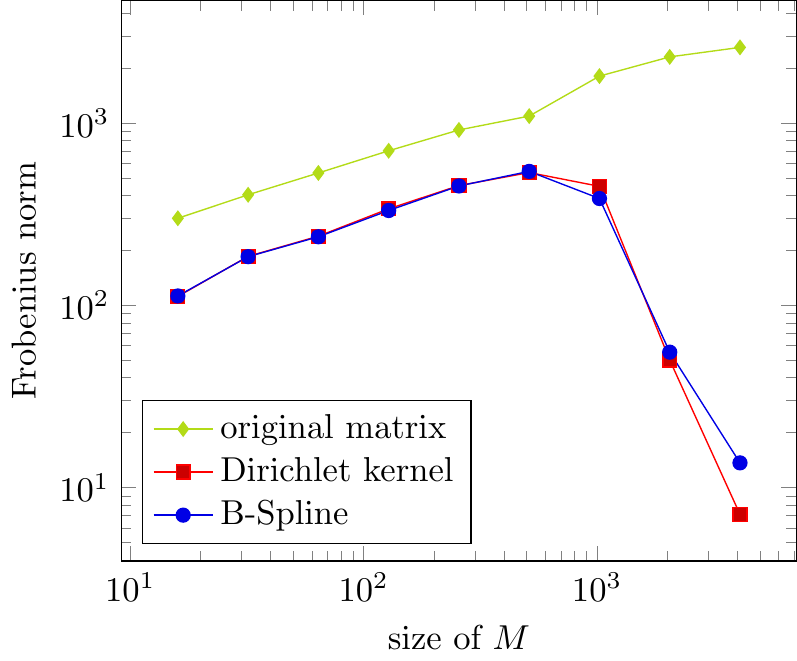}
		\caption{$N=64$}
	\end{subfigure}
	\begin{subfigure}{0.31\textwidth}
		\includegraphics[width=0.9\textwidth]{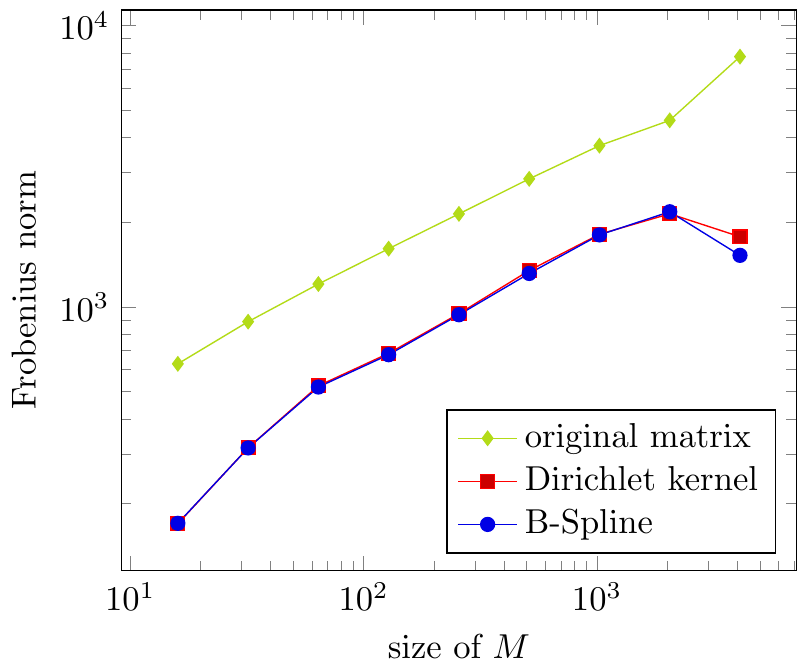}
		\caption{$N=128$}
	\end{subfigure}
	\caption{
		Comparison of Frobenius norms \eqref{eq:matrixnorm_opt_ansatz1} of the original matrix~$\b B^*$ and the optimized matrix~$\b B_{\mathrm{opt}}^*$ generated by Algorithm~\ref{alg:opt_fast_ansatz1} using \textbf{B-Spline} window functions as well as the Dirichlet kernel with \mbox{$m=2$} and \mbox{$\sigma=1.0$} for different numbers of \textbf{Chebyshev nodes} and \mbox{$M=2^c$} with \mbox{$c=4,\dots,12$}.
		\label{fig:matrixnorm_Bspline_tsch_ansatz1}}
\end{figure}
\end{Example}

\begin{Example}
\label{ex:trig_poly_ansatz1}
Secondly, we check if we can retrieve given Fourier coefficients from evaluations of the related trigonometric polynomial~\eqref{eq:trig_polynomial}, cf. Example~\ref{ex:trig_poly_lagrange}.
For \mbox{$\b f = \left(f(x_j)\right)_{j=1}^{N}$} with given nodes~$x_j$ we consider the estimate
\mbox{$\b{\check f} = \left( \check f_k \right)_{k=-\frac M2}^{\frac M2-1} = \frac 1M\, \b D^* \b F^* \b B_{\mathrm{opt}}^* \b f$},
where $\b B_{\mathrm{opt}}^*$ is the output of Algorithm~\ref{alg:opt_fast_ansatz1}.
Here we choose the Dirichlet kernel with \mbox{$\sigma=2.0$} and jittered equispaced nodes~$x_j$.
We consider the absolute and relative errors per node
\begin{equation}
\label{eq:errors_per_node_underdet}
	\frac{e_r^{\mathrm{abs}}}{N} = \frac 1N \|\b A\b{\check f}-\b f\|_{r}
	\quad\text{ and }\quad
	\frac{e_r^{\mathrm{rel}}}{N} = \frac{\|\b A\b{\check f}-\b f\|_{r}}{N\,\|\b f\|_{r}}
\end{equation}
for $r\in\{2,\infty\}$. 
As a first experiment we use \mbox{$N=2^c$} with \mbox{$c=1,\dots,12$}, \mbox{$M=4N$}, and \mbox{$m=4$}.
In a second experiment we fix \mbox{$N=512$} and \mbox{$M=2048$} and the cut-off parameter $m$ shall be chosen \mbox{$m=c$} with \mbox{$c=4,\dots,12$}.
The corresponding results are depicted in Figure~\ref{fig:errors_trig_poly_ansatz1}.
Having a look at the errors per node for growing $N$, see (a), we observe that the errors are worse if we consider very small sizes of $N$.
Otherwise, we recognize that these errors get smaller for large sizes of $N$.
In (b) we can see that for fixed $N$ these errors remain quite stable when tuning the cut-off parameter $m$.
\begin{figure}[h]
	\centering
	\captionsetup[subfigure]{justification=centering}
	\begin{subfigure}{0.4\textwidth}
			\includegraphics[width=0.85\textwidth]{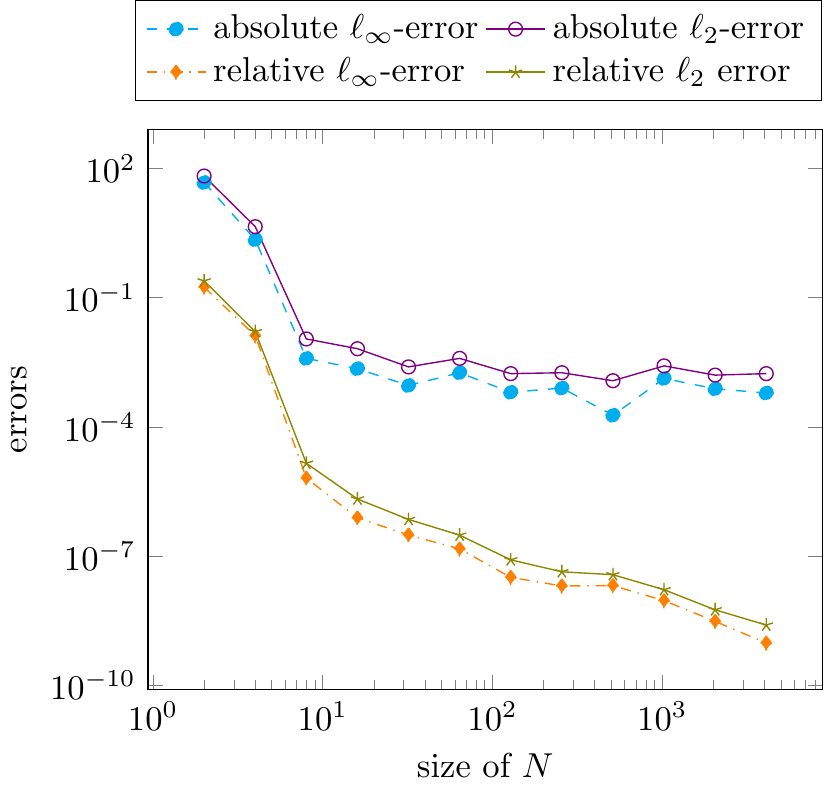}
			\caption{$N=2^c,\, c=1,\dots,12,$ \\ $M=4N$ and\, $m=4$.}
	\end{subfigure}
	\begin{subfigure}{0.4\textwidth}
			\includegraphics[width=0.85\textwidth]{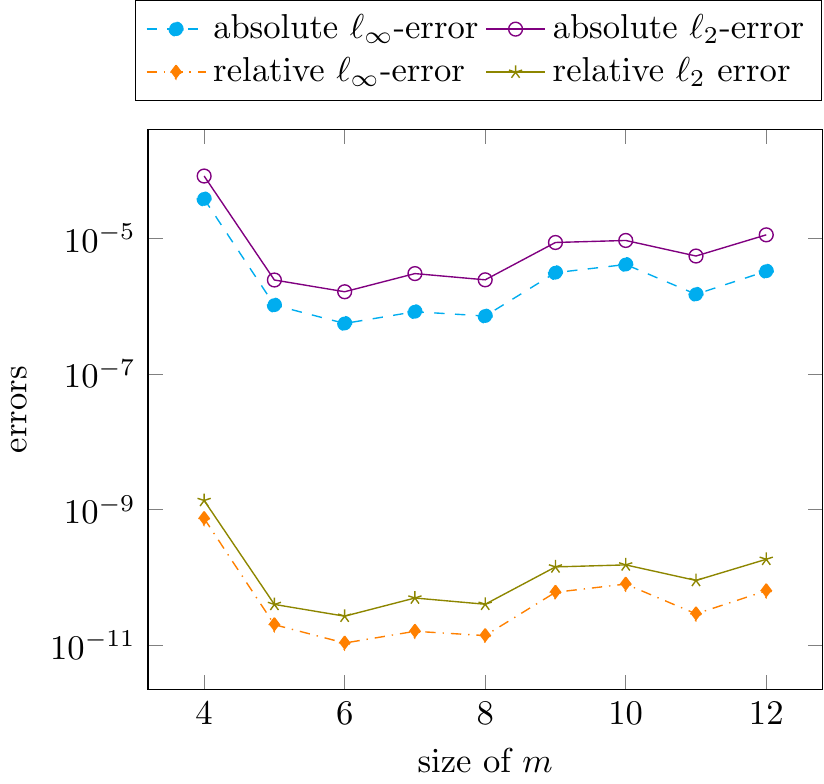}
			\caption{$N=512, M=2048$ \\ and $m=4,\dots,12$.}
	\end{subfigure}
	\caption{Comparison of the errors per node~\eqref{eq:errors_per_node_underdet} of reconstructed Fourier coefficients generated by Algorithm~\ref{alg:opt_fast_ansatz1} using the Dirichlet kernel and $\sigma=2.0$ for jittered equispaced nodes.
		\label{fig:errors_trig_poly_ansatz1}}
\end{figure}
\ex
\end{Example}

\begin{Remark}
\label{remark:ansatz1_adj}
The second problem~\eqref{eq:problem_infft*} can be solved similarly by searching an approximation of the form 
\mbox{$\b B \b K^* = \b B \b F \b D \b A^* \approx M \b I_N$}.
Therefore, we consider the optimization problem 
\begin{equation*}
	\underset{\b B \in \R^{N\times M_\sigma} \colon \b B\, (2m+1)\text{-sparse }}{\text{Minimize }}\ 
	\| \b B \b K^* - M \b I_N \|_{\textrm F}^2.
\end{equation*}
This is equivalent to the transposed problem
\begin{equation*}
	\underset{\b B \in \R^{N\times M_\sigma} \colon \b B\, (2m+1)\text{-sparse }}{\text{Minimize }}\ 
	\| \b K  \b B^* - M \b I_N \|_{\textrm F}^2,
\end{equation*}
which is what we discussed in Section~\ref{subsubsec:ansatz1} and hence can be solved likewise.
\ex
\end{Remark}

\begin{Example}
\label{ex:trig_poly_ansatz1_adj}
Finally, we discuss the analogs of the examples mentioned above for problem~\eqref{eq:problem_infft*}.
Since it is clear that the optimization problems are equivalent we refer to Example~\ref{ex:matrixnorm_ansatz1} for results with respect to the minimization of the norm. 
Similarly to Example~\ref{ex:trig_poly_ansatz1}, we check if we are able to perform an inverse adjoint NFFT for a trigonometric polynomial~\eqref{eq:trig_polynomial}.
This time we consider the estimate
\mbox{$\b{\tilde f} = \left( \tilde f_j \right)_{j=1}^{N} = \frac 1M\, \b B_{\mathrm{opt}} \b F \b D (\b A^* \b f)$}
of the function values \mbox{$\b f = \left(f(x_j)\right)_{j=1}^{N} = \left(f_j\right)_{j=1}^{N}$}, where $\b B_{\mathrm{opt}}$ is the adjoint of the output of Algorithm~\ref{alg:opt_fast_ansatz1}.
We consider the absolute and relative errors per node
\begin{equation}
\label{eq:errors_per_node_adjoint}
	\frac{e_r^{\mathrm{abs}}}{N} = \frac 1N \|\b{\tilde f}-\b f\|_{r}
	\quad\text{ and }\quad
	\frac{e_r^{\mathrm{rel}}}{N} = \frac{\|\b{\tilde f}-\b f\|_{r}}{N\,\|\b f\|_{r}}
\end{equation}
for $r\in\{2,\infty\}$ and perform the same experiments as in Example~\ref{ex:trig_poly_ansatz1}.
The corresponding results can be found in Figure~\ref{fig:errors_trig_poly_ansatz1_adj}.
There we see quite the same behavior in (a), whereas in (b) the errors get even worse when increasing the cut-off parameter $m$.
\begin{figure}[h]
	\centering
	\captionsetup[subfigure]{justification=centering}
	\begin{subfigure}{0.4\textwidth}
		\includegraphics[width=0.85\textwidth]{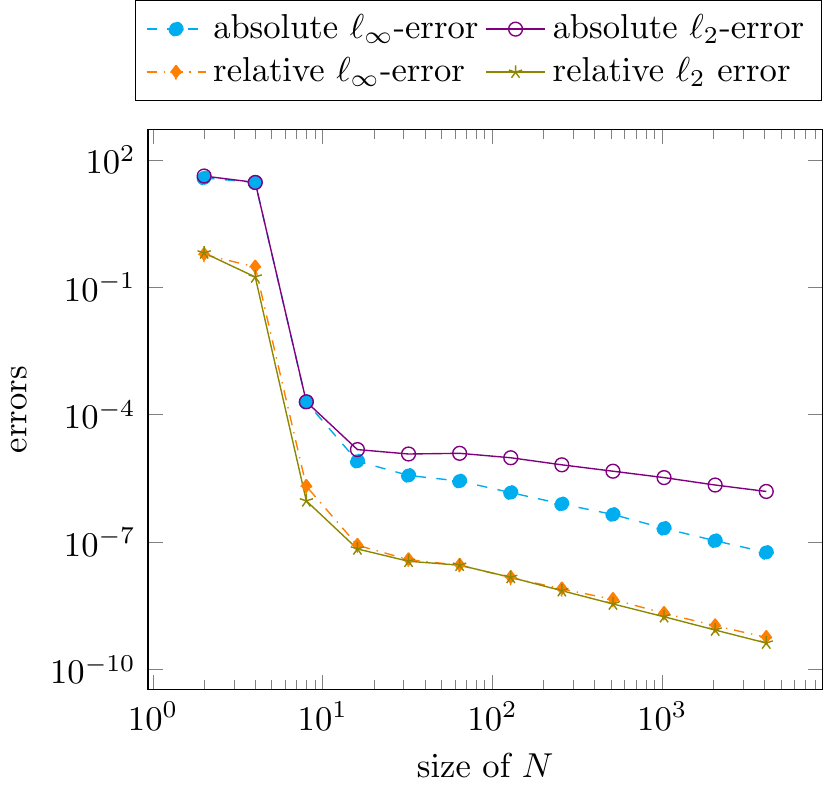}
		\caption{$N=2^c,\, c=1,\dots,12,$ \\ $M=4N$ and\, $m=4$.}
	\end{subfigure}
	\begin{subfigure}{0.4\textwidth}
		\includegraphics[width=0.85\textwidth]{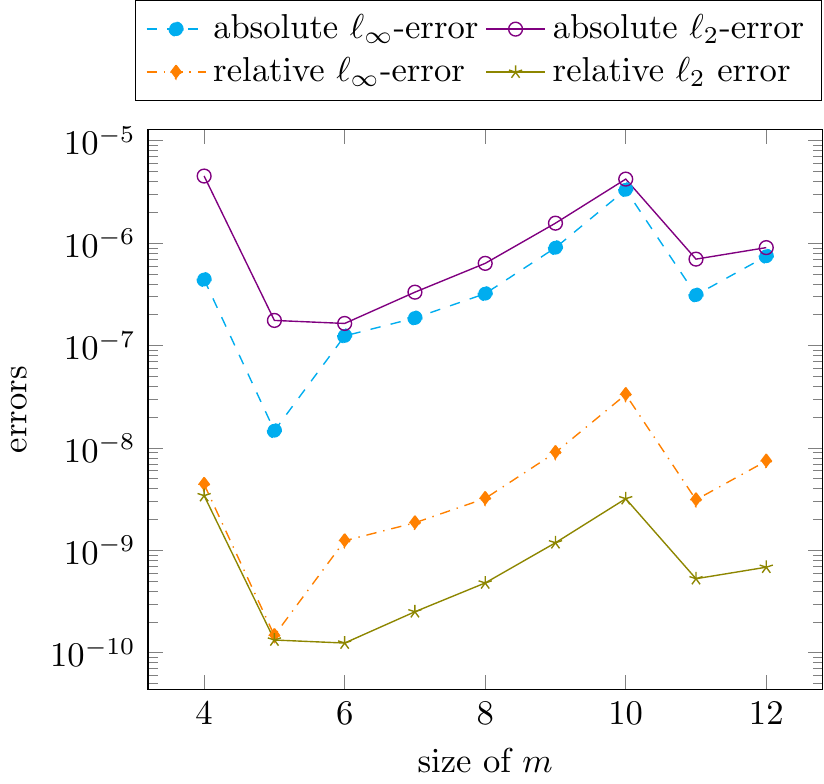}
		\caption{$N=512, M=2048$ \\ and $m=4,\dots,12$.}
	\end{subfigure}
	\caption{Comparison of the errors per node~\eqref{eq:errors_per_node_lagrange} of reconstructed function values generated by Algorithm~\ref{alg:opt_fast_ansatz1} using the Dirichlet kernel and $\sigma=2.0$ for jittered equispaced nodes.
		\label{fig:errors_trig_poly_ansatz1_adj}}
\end{figure}
\ex
\end{Example}

\subsubsection{Inverse NFFT -- overdetermined case\label{subsubsec:ansatz2}}

Previously, in Section~\ref{subsubsec:ansatz1} we studied \mbox{$\b K \b B^*$} and \mbox{$\b B \b K^*$}, where 
\mbox{$\b K = \b A \b D^* \b F^* \in \C^{N\times M_\sigma}$}
and
\mbox{$\b B \in \R^{N \times M_\sigma}$}\!.
There we have seen that the inversion based on the minimization related to these matrices works best for \mbox{$M>N$}, which is the underdetermined case for the inverse NFFT as well as the overdetermined case for the inverse adjoint NFFT.
However, often we are given nonequispaced samples with \mbox{$M<N$} and search a corresponding trigonometric polynomial of degree~$M$.
Hence, we look for another approach which yields the best results in this overdetermined setting \mbox{$M<N$}.
%
To this end, we investigate \mbox{$\b B^* \b K$}.

Initially, we consider the function 
\mbox{$\tilde g(x) = \sum_{j=1}^{N} f_j \, \tilde w_m(x_j-x)$}.
%
Using this we can represent the vector 
\mbox{$\tilde{\b g} \coloneqq \left( \tilde g \hspace{-2.5pt} \left(\tfrac{l}{M_\sigma}\right) \right)_{l=-\frac{M_\sigma}{2}}^{\frac{M_\sigma}{2}-1}$}
by
\mbox{$\tilde{\b g} = \b B^* \b f$}.
Furthermore, we know by \eqref{eq:approx_nfft*} that the adjoint NFFT can be written as 
\mbox{$(\tilde h_k)_{k=-\frac M2}^{\frac M2-1} \eqqcolon \tilde{\b h} = \b D^* \b F^* \b B^* \b f$}
and thereby we have
\mbox{$\tilde{\b h} = \b D^* \b F^* \tilde{\b g}$}.
Now we claim \mbox{$\tilde{\b h} \overset{!}\approx \b{\hat f}$}.
Thus, it follows
\begin{equation*}
	\tilde{\b g} 
	= 
	\b B^* \b f
	=
	\b B^* \b A \b{\hat f}
	\overset{!}\approx
	\b B^* \b A \tilde{\b h}
	=
	\b B^* \b A \b D^* \b F^* \tilde{\b g}
	=
	\b B^* \b K \tilde{\b g},
\end{equation*}
i.\,e., we seek $\b B^*$ as the solution of the optimization problem
\begin{equation*}
	\underset{\b B \in \R^{N\times M_\sigma} \colon \b B\, (2m+1)\text{-sparse }}{\text{Minimize }}\ 
	\left\| \b B^* \b K - \b I_{M_\sigma} \right\|_{\mathrm F}^2.
\end{equation*}
This is equivalent to the transposed problem
\begin{equation}
\label{eq:opt_ansatz2}
	\underset{\b B \in \R^{N\times M_\sigma} \colon \b B\, (2m+1)\text{-sparse }}{\text{Minimize }}\ 
	\left\| \b K^* \b B - \b I_{M_\sigma} \right\|_{\mathrm F}^2.
\end{equation}
By means of definitions~\eqref{eq:matrix_K} and~\eqref{eq:matrix_B} we obtain
\begin{equation}
\label{eq:matrix_B_Phi_T}
	\b K^* \b B 
	=
	\b F \b D \b A^* \b B 
	=
	\left[ 
	\sum_{j=1}^{N} 
	\sum_{k=-\frac M2}^{\frac M2-1} 
	\frac{1}{M_\sigma \hat w(k)} \, \e^{-2\pi\i k \left(x_j-\frac{s}{M_\sigma}\right)} \, 
	\tilde w_m \hspace{-2.75pt} \left(x_j-\tfrac{l}{M_{\sigma}}\right)
	\right]
	_{s,\,l=-\frac{M_\sigma}{2}}
	^{\frac{M_\sigma}{2}-1}.
\end{equation}
Analogously to~\eqref{eq:set_xj}, we define the set
\begin{equation}
\label{eq:set_l}
	I_{M_{\sigma},m}(l) 
	:= 
	\left\{j \in \left\{1, \dots, N \right\}: 
	\exists\, z\in \mathbb Z\ \text{with} -m \leq M_{\sigma} x_j - l + M_{\sigma} z \leq m \right\}.
\end{equation}
Hence, we can rewrite \eqref{eq:opt_ansatz2} by analogy with Section~\ref{subsubsec:ansatz1} as
\begin{equation*}
	\left\| \b F \b D \b A^* \b B - \b I_{M_\sigma} \right\|_{\mathrm F}^2 
	=
	\sum_{l=-\frac{M_\sigma}{2}}^{\frac{M_\sigma}{2}-1} \|\b F \b D \b H_l \b b_l - \b e_l\|_2^2,
\end{equation*}
where 
\begin{equation*}
	\b b_l \coloneqq \bigg( \tilde w_m \hspace{-2.5pt} \left(x_j - \tfrac l{M_{\sigma}}\right) \bigg)
	_{j\in I_{M_{\sigma},m}(l)},
	\quad 
	\b H_l \coloneqq \left( \e^{-2\pi\i k x_j} \right)_{k=-\frac M2,\, j\in I_{M_{\sigma},m}(l)}^{\frac M2-1}
\end{equation*}
and $\b e_l$ denote the columns of the identity matrix~$\b I_{M_\sigma}$.
We obtain, cf.~\eqref{eq:matrix_Kj}, 
\begin{equation}
\label{eq:matrix_Ll}
	\b L_l
	\coloneqq
	\b F \b D \b H_l
	=
	\left[
	\frac 1{M_\sigma} \sum_{k=-\frac M2}^{\frac M2 -1} 
	\frac{1}{\hat{w}(k)}\, \e^{2\pi\i k \left(\frac{s}{M_\sigma}-x_j\right)}
	\right]_{s=-\frac{M_\sigma}{2},\,j \in I_{M_{\sigma},m}(l)}^{\frac{M_\sigma}{2}-1}
	\hspace{-40pt}\in \C^{M_\sigma \times |I_{M_{\sigma},m}(l)|}.
\end{equation} 
Thereby we receive the optimization problems
\begin{equation*}
	\underset{\tilde{\b b}_l \in \R^{2m+1}}{\text{Minimize }}\ 
	\|\b L_l \tilde{\b b}_l - \b e_l\|_2^2, 
	\quad l=-\tfrac{M_\sigma}{2},\dots,\tfrac{M_\sigma}{2}-1.
\end{equation*}
If the matrix
\mbox{$\b L_l \in \C^{M_\sigma \times |I_{M_{\sigma},m}(l)|}$} 
has full rank the solution of~\eqref{eq:opt_ansatz2} is given by
\begin{equation}
\label{eq:solution_ansatz2}
	\tilde{\b b}_l 
	= 
	\left(
	\b L_l ^* 
	\b L_l 
	\right)^{-1}
	\b L_l ^*
	\b e_l, 
	\quad l=-\tfrac{M_\sigma}{2},\dots,\tfrac{M_\sigma}{2}-1.
\end{equation}
This time we cannot tell anything about the dimensions of~$\b L_l$ in general since the size of the set~\mbox{$I_{M_{\sigma},m}(l)$} depends on several parameters.
Having these vectors~$\tilde{\b b}_l$ we can compose the modified matrix~$\b B_{\mathrm{opt}}$, observing that $\tilde{\b b}_l$ only consist of the nonzero entries of $\b B_{\mathrm{opt}}$.
Then the approximation of the Fourier coefficients is given by
\begin{equation}
\label{eq:approx_ansatz2}
	\b{\hat f} 
	\approx
	\tilde{\b h}
	=
	\b D^* \b F^* \b B_{\mathrm{opt}}^* \b f.
\end{equation}
In other words, this approach yields another way to invert the NFFT by also modifying the adjoint NFFT.

Analogously to Section~\ref{subsubsec:ansatz1}, we are able to compute the entries of the matrix~$\b L_l$, see \eqref{eq:matrix_Ll}, by means of an NFFT with the $M$ coefficients
\begin{equation}
\label{eq:fourier_coefficients_fast_ansatz2}
	\hat f_k = \frac{1}{M_\sigma \hat w(k)},
	\quad
	k=-\tfrac M2,\dots,\tfrac M2-1,
\end{equation}
and nodes
\mbox{$y_{s,j} \coloneqq \tfrac{s}{M_\sigma} - x_j,\, s = -\tfrac{M_\sigma}{2},\dots,\tfrac{M_\sigma}{2}-1, \, j \in I_{M_{\sigma},m}(l),$}
which are at most \mbox{$M_\sigma N$} many.
Here we also require only one NFFT by writing the columns of $\b L_l$ one below the other.
The obtained vector including all entries of $\b L_l$ has to be reshaped afterwards.
This leads to the following algorithm.

\begin{algorithm}{Fast optimization of the matrix $\b B$}
	\label{alg:opt_fast_ansatz2}
	For $N \in \N$ let $x_j \in \left[-\frac 12, \frac 12\right), j=1,\dots,N,$ be given nodes as well as $M \in 2\N$, $\sigma \geq 1$ and $M_\sigma = \sigma M$.
	\begin{enumerate}	
		\item Compute $\b g = \b F \b D \b{\hat f}$, cf.~\eqref{eq:matrix_D} and \eqref{eq:matrix_F}, with $\hat f_k$ in \eqref{eq:fourier_coefficients_fast_ansatz2}.
		\hfill $\mathcal O(M\log M)$
		\item
		For $l = -\frac{M_\sigma}{2}, \dots, \frac{M_\sigma}{2}-1$: 
		\begin{itemize}
			\item[] \hspace{-0.6cm} Determine the set $I_{M_{\sigma},m}(l)$, cf. \eqref{eq:set_l}.
			\hfill $\mathcal O(N)$
			\item[] \hspace{-0.6cm} Perform $\b B \b g$, cf.~\eqref{eq:matrix_B}, for the vector of nodes
			\begin{equation*}
			\b y \coloneqq \left( \b y_1^T, \dots, \b y_s^T \right)^T
			\end{equation*} 
			\hspace{-0.6cm} for $\b y_n$ being the columns of the matrix $\b Y \coloneqq (y_{l,j})_{l=-\frac{M_\sigma}{2},\, j \in I_{M_{\sigma},m}(l)}^{\frac{M_\sigma}{2}-1}$.
			\hfill $\mathcal O(N)$
			\item[] \hspace{-0.6cm} Reshape the obtained vector 
			into the matrix $\b L_l \in \C^{M_\sigma \times |I_{M_{\sigma},m}(l)|}$.
			\hfill $\mathcal O(N)$
			\item[] \hspace{-0.6cm} Solve the normal equations for \(\b L_l\), 
			cf. \eqref{eq:solution_ansatz2}. \hfill \mbox{$\mathcal O(N^3+N^2 M)$}
		\end{itemize}
		\item Compose $\b B_{\mathrm{opt}}$ column-wise of the vectors $\tilde{\b b}_l$ observing the periodicity.
		\hfill $\mathcal O(M)$
	\end{enumerate}
	\vspace{2pt} 
	Output: optimized matrix $\b B_{\mathrm{opt}}$ 
	\vspace{2pt} \\
	Complexity: $\mathcal O(N^2 M^2 + N^3 M)$
\end{algorithm}

\begin{Remark}
\label{remark:alg:uniform_nodes}
If we assume the nodes are somewhat uniformly distributed, like for instance jittered equispaced nodes, we can get rid of the complexity related to $N$ and end up with arithmetic costs of \mbox{$\mathcal O(M^2)$}.
\ex
\end{Remark}

\begin{Remark}
\label{remark:Dirichlet_ansatz2}
It is also possible to simplify the computation of~$\b L_l$ by incorporating the Dirichlet kernel~\eqref{eq:kernel_Dirichlet}, i.\,e., we set \mbox{$\hat w(k)=1$} for all \mbox{$k=-\frac M2+1, \dots, \frac M2-1,$} and the last nonzero entry \mbox{$\frac{1}{\hat w(\frac M2)}$} of the matrix~$\b D^*$ is set to zero.
Hence, the entries of the matrix 
\begin{equation*}
	\b L_l
	= 
	\left[
	\frac 1{M_\sigma} \,
	D_{\frac M2-1}\left(\tfrac{l}{M_\sigma}-x_j\right)
	\right]_{l=-\frac{M_\sigma}{2},\,j \in I_{M_{\sigma},m}(l)}^{\frac{M_\sigma}{2}-1}
\end{equation*}
can explicitly be stated and therefore the term \mbox{$M\log M$} in the computational costs of Algorithm~\ref{alg:opt_fast_ansatz2} can be eliminated.
Nevertheless, even if we assume uniformly distributed nodes as in Remark~\ref{remark:alg:uniform_nodes}, we remain with arithmetic costs of \mbox{$\mathcal O(M^2)$}.
\hfill \ex
\end{Remark}

\begin{Remark}
This algorithm using a window function or the Dirichlet kernel is part of the software package NFFT 3.5.1, see \cite[./matlab/infft1d]{nfft3}.
\ex
\end{Remark}

\begin{Example}
\label{ex:matrixnorm_ansatz2}
As in Example \ref{ex:matrixnorm_ansatz1} we verify at first that the optimization was successful.
On that account, we compare the norms
\begin{equation}
\label{eq:matrixnorm_opt_ansatz2}
	\left\| \b F \b D \b A^* \b B - \b I_{M_\sigma} \right\|_{\mathrm F}
	\quad\text{ and }\quad
	\| \b F \b D \b A^* \b B_{\mathrm{opt}} - \b I_{M_\sigma} \|_{\mathrm F},
\end{equation}
where $\b B$ denotes the original matrix from the NFFT and $\b B_{\mathrm{opt}}$ the optimized matrix generated by Algorithm~\ref{alg:opt_fast_ansatz2}.
Although our method is attributed to the overdetermined setting, this again means no restriction.
Therefore, also the underdetermined setting is tested.
\begin{itemize}
	\item[(i)] 
	Again we examine at first jittered equispaced nodes, see~\eqref{eq:jittered_nodes}.
	We choose \mbox{$M=128$} and consider the norms~\eqref{eq:matrixnorm_opt_ansatz2} for \mbox{$N=2^c$} nodes with \mbox{$c=2,\dots,14$}.
	In order to compute the NFFT in Algorithm~\ref{alg:opt_fast_ansatz2} we choose the Kaiser-Bessel window, an oversampling of \mbox{$\sigma_2=2.0$} and the cut-off parameter~\mbox{$m_2=2m$} to achieve results comparable to Example~\ref{ex:matrixnorm_ansatz1}.
	However, one could also choose a larger cut-off to receive more accuracy for growing~$N$.
	
	In Figure~\ref{fig:matrixnorm_Bspline_jittered_ansatz2} one can find the comparison of the norms~\eqref{eq:matrixnorm_opt_ansatz2} for different values of $m$ and $\sigma$ for $\b B_{\mathrm{opt}}$ generated using B-Splines as well as the Dirichlet kernel mentioned in Remark~\ref{remark:Dirichlet_ansatz2}. 
	It can be seen that for \mbox{$\sigma=1.0$} the minimization was very successful especially for large $N$ compared to~$M$.
	For \mbox{$N<M$} the minimization was not successful.
	Similarly to Example~\ref{ex:matrixnorm_ansatz1}, this results from the fact that the corresponding matrix \mbox{$\b F \b D \b A^* \b B$} is of low rank.
	Therefore, Algorithm~\ref{alg:opt_fast_ansatz2} is specially attributed to the overdetermined case.
	
	Having a look at the graphs with high oversampling we recognize that the norms of the optimized matrices remain stable for all sizes of~$N$.
	Thus, also for this method the optimization seems not to work for high oversampling.
	
	While the computational costs could not be scaled down, we see in Figure~\ref{fig:runtime_Dirichlet_ansatz2} that using the Dirichlet kernel reduced the run-time for all sizes of~$N$.
	Results for other window functions are omitted since they show the same behavior.	
	\item[(ii)] 
	Next we repeat the example using Chebyshev nodes, cf.~\eqref{eq:tschebyscheff_nodes}.
	The cor\-res\-pon\-ding results for B-Splines can be found in Figure~\ref{fig:matrixnorm_Bspline_tsch_ansatz2}.
	We recognize that these graphs look similar to Figure~\ref{fig:matrixnorm_Bspline_jittered_ansatz2}, in contrast to Example~\ref{ex:matrixnorm_ansatz1} where the optimization for Chebyshev nodes was quite difficult. 
	\ex
\end{itemize}
\begin{figure}[p]
	\centering
	\begin{subfigure}{0.31\textwidth}
		\includegraphics[width=0.9\textwidth]{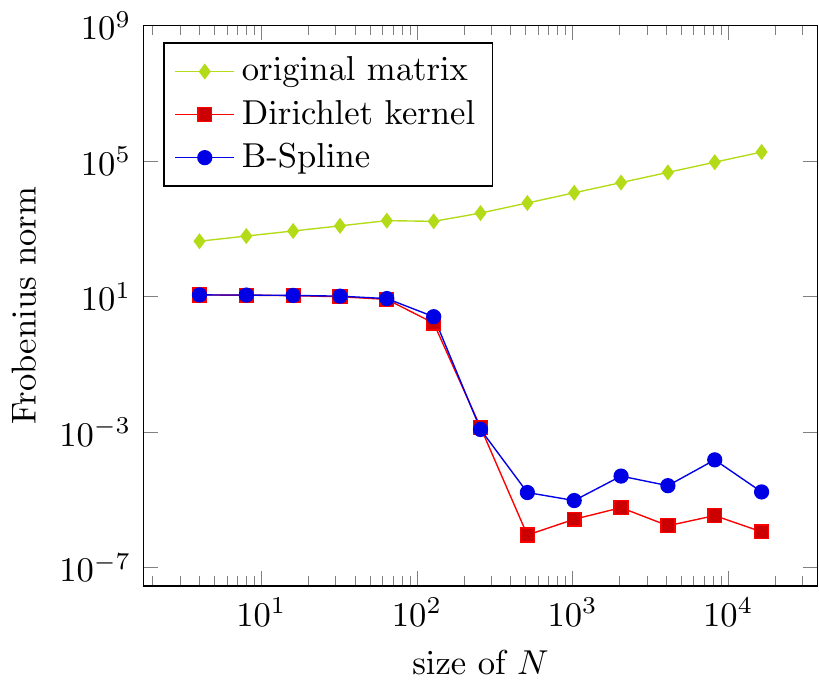}
		\caption{$m=2$ and $\sigma=1.0$}
	\end{subfigure}
	\begin{subfigure}{0.31\textwidth}
		\includegraphics[width=0.9\textwidth]{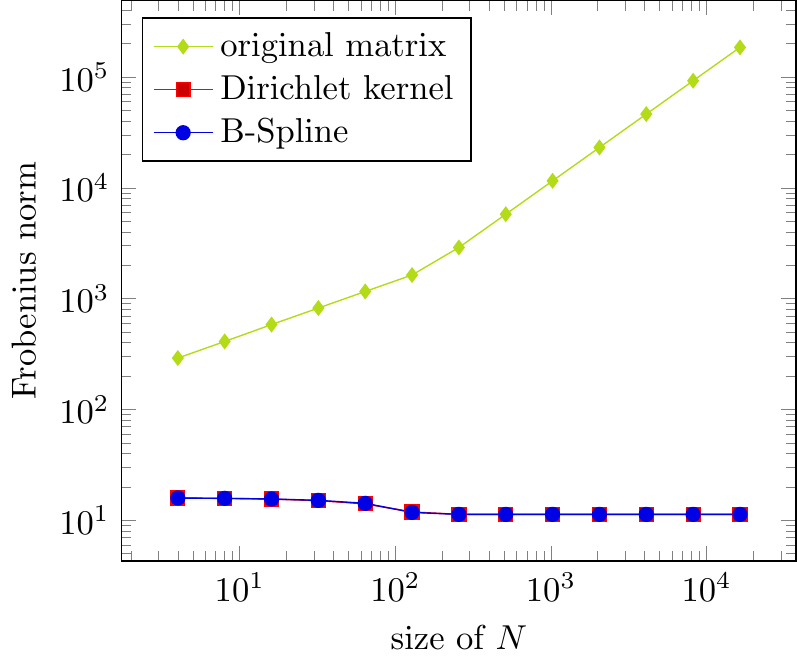}
		\caption{$m=2$ and $\sigma=2.0$}
	\end{subfigure}
	\begin{subfigure}{0.31\textwidth}
		\includegraphics[width=0.9\textwidth]{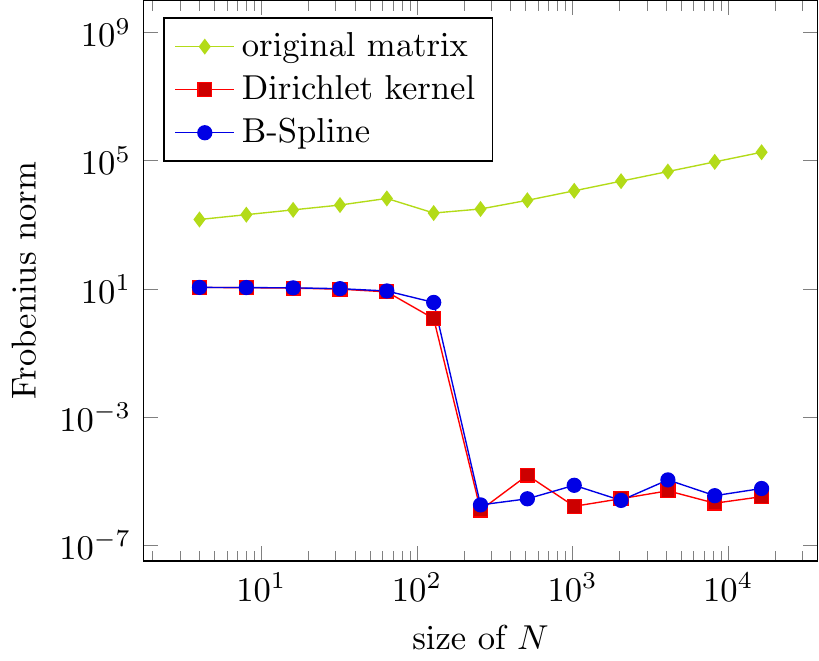}
		\caption{$m=4$ and $\sigma=1.0$}
	\end{subfigure}
	\caption{
		Comparison of Frobenius norms \eqref{eq:matrixnorm_opt_ansatz2} of the original matrix~$\b B$ and the optimized matrix~$\b B_{\mathrm{opt}}$ generated by Algorithm~\ref{alg:opt_fast_ansatz2} using \textbf{B-Spline} window functions as well as the Dirichlet kernel for \mbox{$M=128$} and \mbox{$N=2^c$} \textbf{jittered equispaced nodes} with \mbox{$c=2,\dots,14$}.
		\label{fig:matrixnorm_Bspline_jittered_ansatz2}}
\end{figure}
%
%
\begin{figure}[p]
	\centering
	\includegraphics[width=0.38\textwidth]{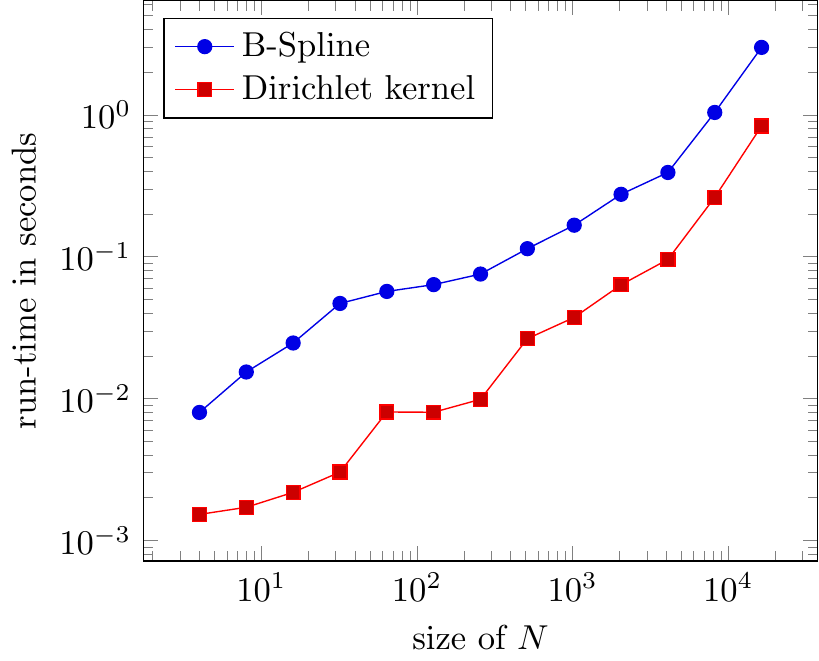}
	\caption{
		Comparison of run-times of Algorithm~\ref{alg:opt_fast_ansatz2} using the \mbox{\textbf{B-Spline}} as well as the Dirichlet kernel with $m=2$ and $\sigma=1.0$ for \mbox{$M=128$} and \mbox{$N=2^c$} \textbf{jittered equispaced nodes} with \mbox{$c=2,\dots,14$}.
		\label{fig:runtime_Dirichlet_ansatz2}}
\end{figure}
	\begin{figure}[p]
		\centering
		\begin{subfigure}{0.31\textwidth}
			\includegraphics[width=0.9\textwidth]{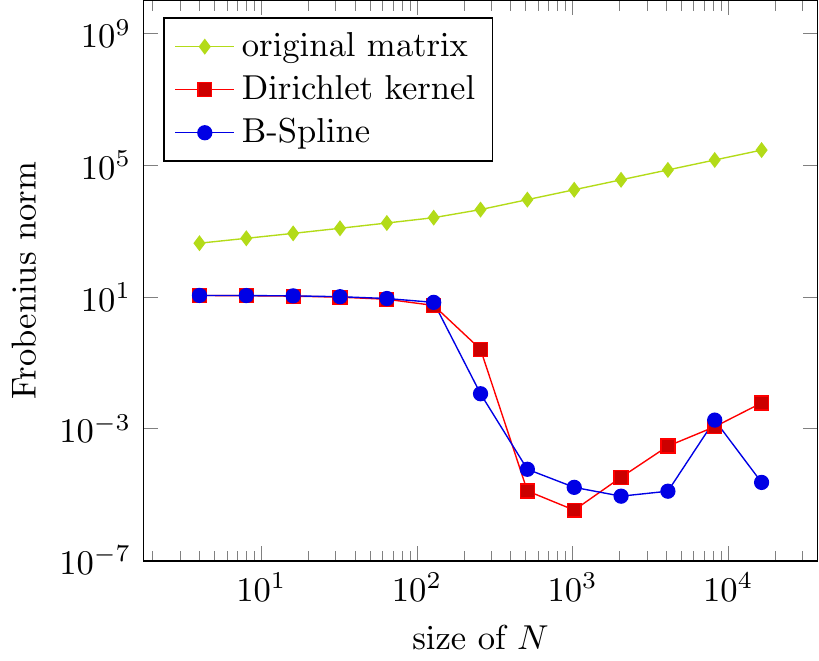}
			\caption{$m=2$ and $\sigma = 1.0$}
		\end{subfigure}
		\begin{subfigure}{0.31\textwidth}
			\includegraphics[width=0.9\textwidth]{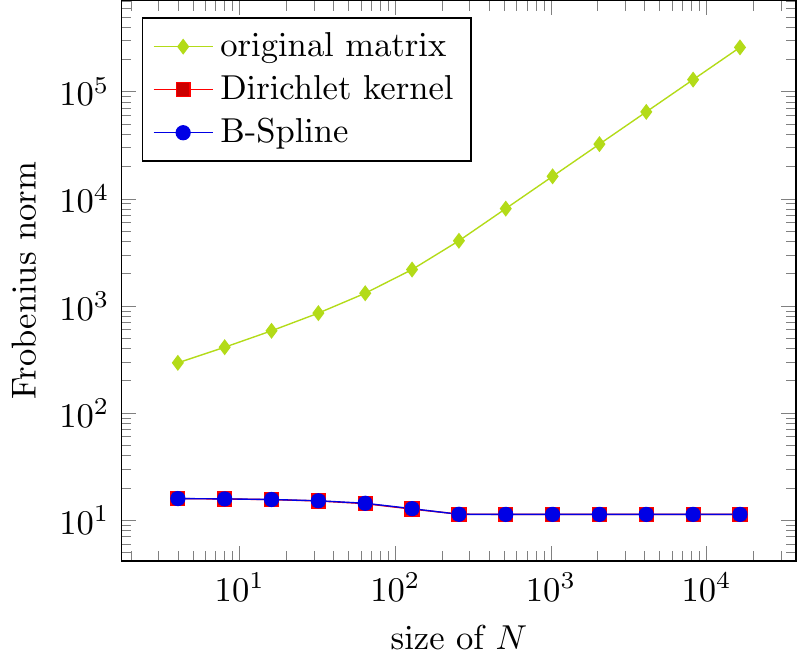}
			\caption{$m=2$ and $\sigma = 2.0$}
		\end{subfigure}
		\begin{subfigure}{0.31\textwidth}
			\includegraphics[width=0.9\textwidth]{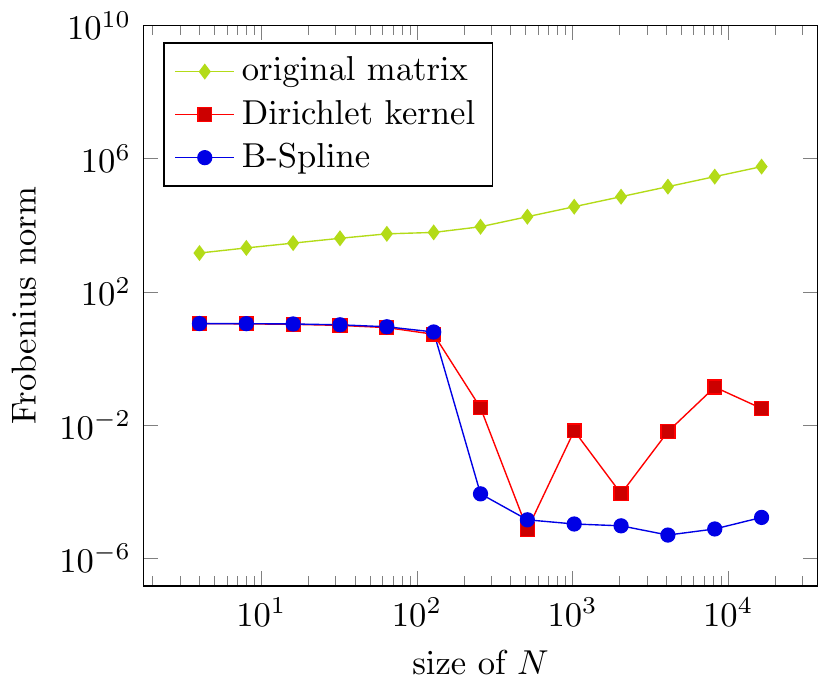}
			\caption{$m=4$ and $\sigma = 1.0$}
		\end{subfigure}
		\caption{
			Comparison of Frobenius norms \eqref{eq:matrixnorm_opt_ansatz2} of the original matrix~$\b B$ and the optimized matrix~$\b B_{\mathrm{opt}}$ generated by Algorithm~\ref{alg:opt_fast_ansatz2} using \textbf{B-Spline} window functions as well as the Dirichlet kernel for \mbox{$M=128$} and \mbox{$N=2^c$} \textbf{Chebyshev nodes} with \mbox{$c=2,\dots,14$}.
			\label{fig:matrixnorm_Bspline_tsch_ansatz2}}
	\end{figure}
\end{Example}

\begin{Remark}
\label{remark:toeplitz_approach}
Another approach for computing an inverse NFFT in the setting $N\geq M$ can be obtained by using the fact that $\b A^* \b A$ is of Toeplitz structure.
To this end, the Gohberg-Semencul formula, see \cite{HeRo84}, can be used to solve the normal equations $\b A^* \b A \b{\hat f}=\b A^* \b f$ exactly by analogy with \cite{AvShSh16}.
Therefore, the inverse NFFT consists of two steps: an adjoint NFFT applied to $\b f$ and the multiplication with the inverse of $\b A^* \b A$, which can be realized by means of 8 FFTs.
The computation of the components of the Gohberg-Semencul formula can be seen as a precomputational step.

However, even if this algorithm is exact and therefore yields better results than our approach from Section~\ref{subsubsec:ansatz2} there is no exact generalization to higher dimensions since there is no generalization of the Gohberg-Semencul formula to dimensions $>1$.
This approach can only be used for approximations on very special grids, such as linogramm grids, utilizing the given specific structure, see \cite{AvCoDoIsSh08, AvShSh16}.
\ex
\end{Remark}

\begin{Example}
\label{ex:trig_poly_ansatz2}
As in Example~\ref{ex:trig_poly_ansatz1} we consider a trigonometric polynomial.
For \mbox{$\b f = \left(f(x_j)\right)_{j=1}^{N}$} with given nodes~$x_j$ we consider the estimate
\mbox{$\b{\check f} = \left( \check f_k \right)_{k=-\frac M2}^{\frac M2-1} = \b D^* \b F^* \b B_{\mathrm{opt}}^* \b f$},
where $\b B_{\mathrm{opt}}^*$ is the outcome of Algorithm~\ref{alg:opt_fast_ansatz2}, and compare it to the given function values $\b f$.
Again, we choose the Dirichlet kernel with \mbox{$\sigma=2.0$} and jittered equispaced nodes~$x_j$.
We consider the absolute and relative errors per node \eqref{eq:errors_per_node_lagrange}
for $r\in\{2,\infty\}$.
As a first experiment we use \mbox{$M=2^c$} with \mbox{$c=1,\dots,12$}, \mbox{$N=4M$}, and \mbox{$m=4$}.
In a second experiment we fix \mbox{$N=2048$} and \mbox{$M=512$} and the cut-off parameter $m$ shall be chosen \mbox{$m=c$} with \mbox{$c=4,\dots,12$}.
The corresponding results are displayed in Figure~\ref{fig:errors_trig_poly_ansatz2}.
We recognize that the errors per node for growing $M$, see (a), are worse if we consider very small sizes of $N$ but decrease for large sizes of $M$.
In (b) we can see that for fixed $M$ these errors remain quite stable when tuning the cut-off parameter $m$.
One could also consider the Toeplitz approach described in Remark~\ref{remark:toeplitz_approach}. There we observed stable absolute errors of size $10^{-15}$ and stable relative errors of size $10^{-18}$.
\begin{figure}[h]
	\centering
	\captionsetup[subfigure]{justification=centering}
	\begin{subfigure}{0.4\textwidth}
			\includegraphics[width=0.85\textwidth]{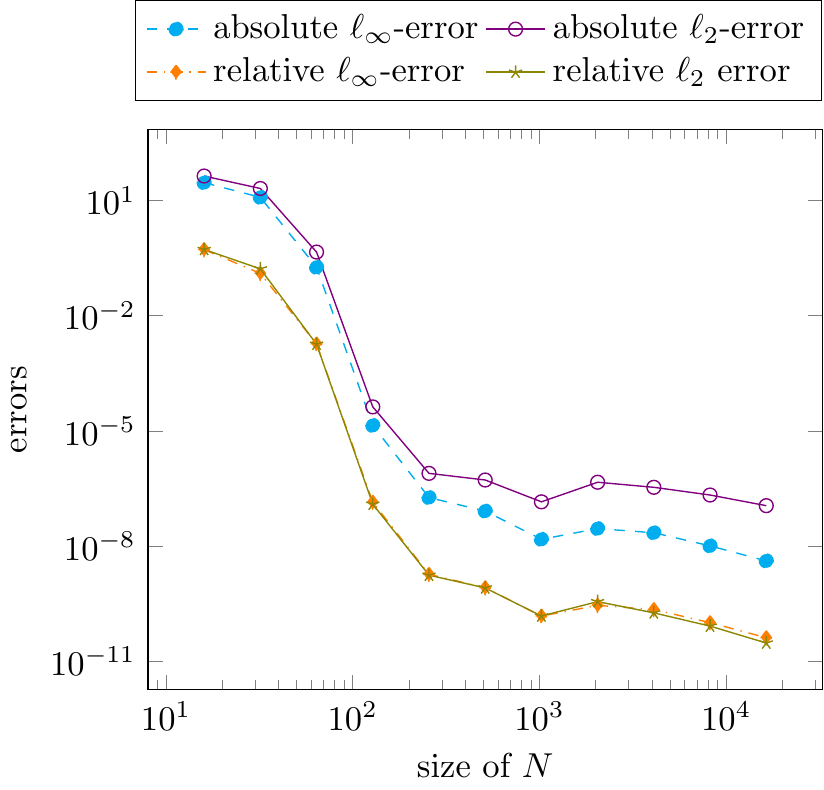}
			\caption{$M=2^c,\, c=1,\dots,12,$ \\ $N=4M$ and\, $m=4$.}
	\end{subfigure}
	\begin{subfigure}{0.4\textwidth}
			\includegraphics[width=0.85\textwidth]{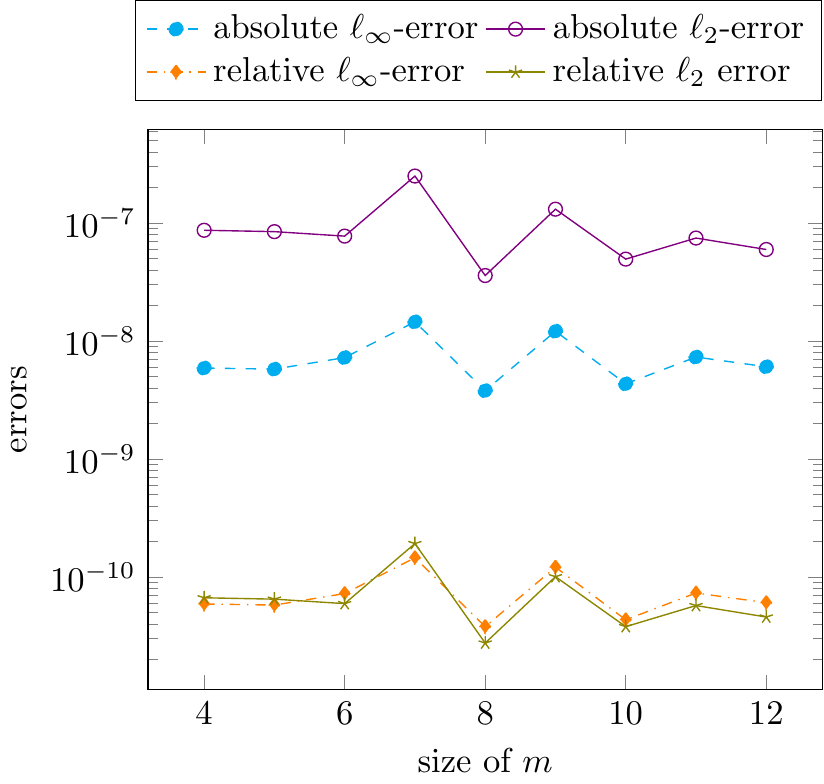}
			\caption{$N=2048, M=512$ \\ and $m=4,\dots,12$.}
	\end{subfigure}
	\caption{Comparison of the errors per node~\eqref{eq:errors_per_node_lagrange} of reconstructed Fourier coefficients generated by Algorithm~\ref{alg:opt_fast_ansatz2} using the Dirichlet kernel and $\sigma=2.0$ for jittered equispaced nodes.
		\label{fig:errors_trig_poly_ansatz2}}
\end{figure}
\ex
\end{Example}

\begin{Remark}
\label{remark:ansatz2_adj}
Having a look at the remaining matrix product \mbox{$\b K^* \b B$} we recognize that the corresponding optimization problem
\begin{equation*}
	\underset{\b B \in \R^{N\times M_\sigma} \colon \b B\, (2m+1)\text{-sparse }}{\text{Minimize }} \ 
	\left\| \b K^* \b B - \b I_{M_\sigma} \right\|_{\mathrm F}^2
\end{equation*}
was already solved to find a solution for the transposed problem
\begin{equation*}
	\underset{\b B \in \R^{N\times M_\sigma} \colon \b B\, (2m+1)\text{-sparse }}{\text{Minimize }} \ 
	\left\| \b B^* \b K - \b I_{M_\sigma} \right\|_{\mathrm F}^2.
\end{equation*}
Hence, it is merely left to examine the current approximation. Due to the minimization problem we have
\mbox{$\b F \b D \b A^* \b B = \b K^* \b B \approx \b I_{M_\sigma}$}.
Because \mbox{$\b B \in \R^{N \times M_\sigma}$} is rectangular and therefore not invertible 
we multiply by a right-inverse of~$\b B$, i.\,e., a matrix \mbox{$\b B' \in \R^{M_\sigma \times N}$} that holds \mbox{$\b B \b B' = \b I_N$}, and receive
\mbox{$\b F \b D \b A^* \approx \b B'$}.
Multiplying by a vector~$\b f$ yields
\mbox{$\b F \b D \b A^* \b f \approx \b B' \b f$},
which can be written by means of \mbox{$\b A^* \b f = \b h$} as 
\mbox{$\b F \b D \b h	\approx	\b B' \b f$}.
Finally, we multiply left-hand by~$\b B$, which results in the approximation
\mbox{$\b B \b F \b D \b h \approx \b f$}
and thus provides another method to invert the adjoint NFFT by modifying the NFFT.
\ex
\end{Remark}

\begin{Example}
\label{ex:trig_poly_ansatz2_adj}
Finally, we consider the analog of Example~\ref{ex:trig_poly_ansatz1_adj}.
To this end, we exchange the estimate by
\mbox{$\b{\tilde f} = \b B_{\mathrm{opt}} \b F \b D (\b A^* \b f)$}, 
where $\b B_{\mathrm{opt}}$ is the adjoint of the outcome of Algorithm~\ref{alg:opt_fast_ansatz2}.
We conduct the same experiments as in Example~\ref{ex:trig_poly_ansatz2} but consider the errors per node~\eqref{eq:errors_per_node_adjoint}.
The corresponding results can be found in Figure~\ref{fig:errors_trig_poly_ansatz2_adj}.
There we see that our optimization was not successful in the case $N>M$ since there is no reasonable chance to approximate the function values in any of the tested settings.
\begin{figure}[h]
	\centering
	\captionsetup[subfigure]{justification=centering}
	\begin{subfigure}{0.4\textwidth}
		\includegraphics[width=0.85\textwidth]{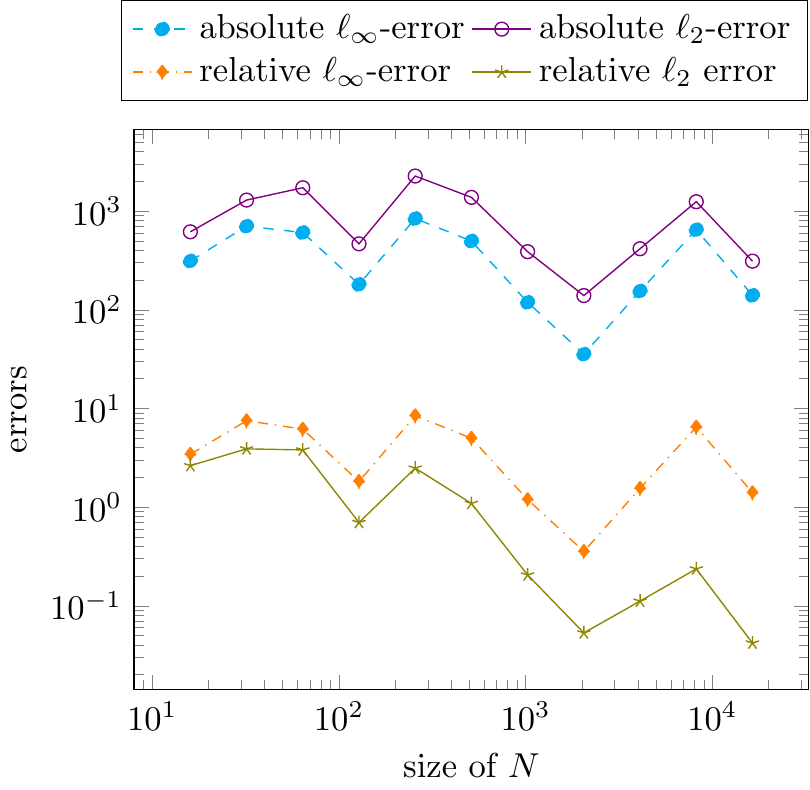}
		\caption{$M=2^c,\, c=1,\dots,12,$ \\ $N=4M$ and\, $m=4$.}
	\end{subfigure}
	\begin{subfigure}{0.4\textwidth}
		\includegraphics[width=0.85\textwidth]{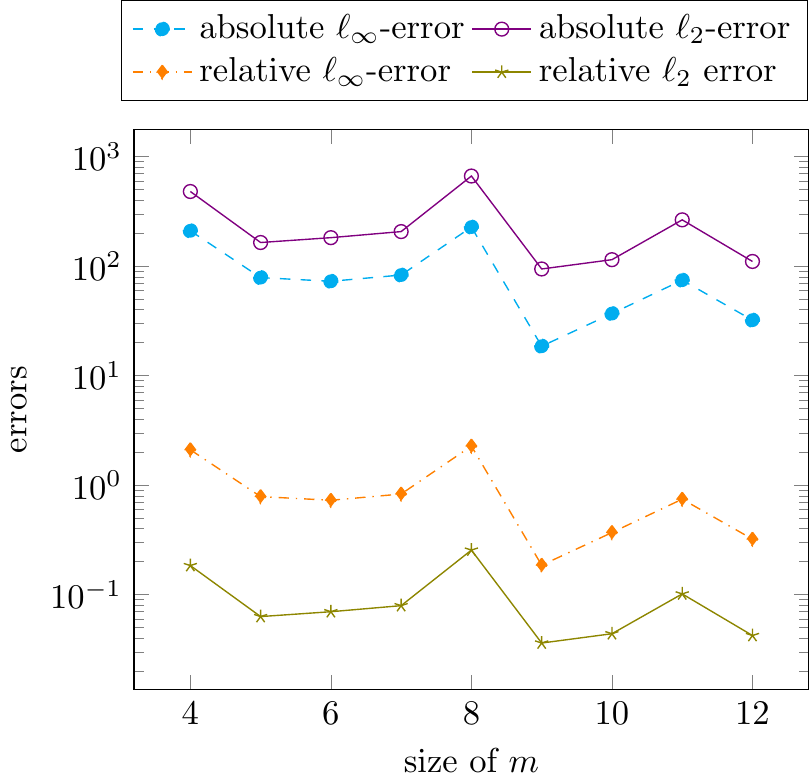}
		\caption{$N=2048, M=512$ \\ and $m=4,\dots,12$.}
	\end{subfigure}
	\caption{Comparison of the errors per node~\eqref{eq:errors_per_node_lagrange} of reconstructed function values generated by Algorithm~\ref{alg:opt_fast_ansatz2} using the Dirichlet kernel and $\sigma=2.0$ for jittered equispaced nodes.
		\label{fig:errors_trig_poly_ansatz2_adj}}
\end{figure}
\ex
\end{Example}

\section{Frames\label{sec:frames}}

During the last few decades the popularity of frames rose rapidly and more and more articles are concerned with this topic.
Recently, an approach was published in~\cite{GeSo14} connecting frame approximation to the adjoint NFFT.
Thus, in this section we consider the concept of frames and discuss an approach for inverting the NFFT based on \cite{GeSo14}.
Besides the basic information about the approximation of the inverse frame operator, a link to the methods explained in Section~\ref{subsec:rect} is provided.

\subsection{Approximation of the inverse frame operator\label{subsec:approx}}

First of all, we sum up the main idea of frames and frame approximation, basically adapted from~\cite{GeSo14} and~\cite{Chr16}.

\begin{definition}
Let $\mathcal H$ be a separable Hilbert space with inner product \mbox{$\langle\cdot,\cdot\rangle$}.
Then a sequence \mbox{$\{\varphi_j\}_{j=1}^{\infty}\subset\mathcal H$} is called \textbf{frame} if there exist constants \mbox{$A,B>0$} such that
\begin{equation*}
	A \|f\|^2 \leq \sum_{j=1}^{\infty} |\langle f,\varphi_j \rangle|^2 \leq B \|f\|^2 \quad \forall f \in \mathcal H.
\end{equation*}
The operator
\mbox{$S \colon \mathcal H \to \mathcal H, Sf = \sum_{j=1}^{\infty} \langle f, \varphi_j \rangle \varphi_j$},
is named the \textbf{frame operator}.
\ex
\end{definition}

Given this definition we can already state one of the most important results in frame theory, the so-called frame decomposition.
If 
\mbox{$\{\varphi_j\}_{j=1}^{\infty}$}
is a frame with frame operator~$S$, then
\begin{equation}
\label{eq:frame_decomp}
f = \sum_{j=1}^{\infty} \langle f, S^{-1} \varphi_j \rangle \varphi_j
= \sum_{j=1}^{\infty} \langle f, \varphi_j \rangle S^{-1} \varphi_j \quad \forall f \in \H.
\end{equation}
In other words, every element of $\H$ can be represented as a linear combination of the elements of the frame, which is a property similar to an orthonormal basis. 
Though, to apply~\eqref{eq:frame_decomp} it is necessary to state the inverse operator~$S^{-1}$ explicitly.
However, this is usually difficult (or even impossible).
Hence, it is necessary to be able to approximate~$S^{-1}$.
For this purpose, we use the method from~\cite{GeSo12} by analogy with~\cite{GeSo14}, which is based on so-called admissible frames, see~\cite[Definition 1]{GeSo14}.
%
%
%

We suppose 
\mbox{$\{\psi_l\}_{l=-\infty}^{\infty}$}
is an admissible frame with respect to 
\mbox{$\{\varphi_j\}_{j=1}^{\infty}$}.
As shown in \cite{GeSo12}, the dual frame 
\mbox{$\{S^{-1} \varphi_j\}_{j=1}^{\infty}$}
can then be approximated by
\begin{equation}
\label{eq:approx_dual}
	S^{-1} \varphi_j \approx \tilde \varphi_j \coloneqq \sum_{l=-\frac{M_\sigma}{2}}^{\frac{M_\sigma}{2}-1} p_{l,j} \, \psi_l, \quad j = 1, \dots, N,
\end{equation}
where
\mbox{$\b \Phi ^\dagger \eqqcolon \left[ p_{l,j} \right]_{l=-\frac{M_\sigma}{2},\,j=1}^{\frac{M_\sigma}{2}-1,\;N}$}
is the Moore-Penrose pseudoinverse of the matrix
\begin{equation}
\label{eq:matrix_Phi}
	\b \Phi \coloneqq \left[ \langle \varphi_j, \psi_l \rangle \right]_{j=1,\, l=-\frac{M_\sigma}{2}}^{N,\  \frac{M_\sigma}{2}-1}.
\end{equation}
In so doing, the matrix dimensions have to fulfill the condition
\mbox{$N \geq M_\sigma + c M_\sigma^{\frac{1}{2s-1}}$}.
Given this approximation of the dual frame, inserting \eqref{eq:approx_dual} in \eqref{eq:frame_decomp} and cutting off the infinite sum yields the approximation
\begin{equation}
\label{eq:approx_function}
	f
	\approx
	\tilde f 
	\coloneqq 
	\sum_{j=1}^{N} \sum_{l=-\frac{M_\sigma}{2}}^{\frac{M_\sigma}{2}-1} \langle f, \varphi_j \rangle \, p_{l,j} \, \psi_l.
\end{equation}

\subsection{Linking the frame-theoretical approach to the iNFFT\label{subsec:link}}

Now we aim to find a link between the frame approximation~\eqref{eq:approx_function} and the iNFFT from Section~\ref{subsec:rect}.
To this end, we consider a discrete version of the frames recommended in~\cite{GeSo14}, i.\,e.,
\begin{equation}
\label{eq:frames_phi_psi}
	\{\varphi_j(k) \coloneqq \e^{-2\pi\i k x_j}, \, j \in \N \}
	\quad\text{ and }\quad
	\left\{\psi_l(k) \coloneqq \frac{\e^{-2\pi\i kl/M_{\sigma}}}{M_\sigma \hat w(-k)}, \, l \in \Z \right\}
\end{equation}
for \mbox{$k \in \Z$}, where \mbox{$x_j \in [-\frac 12, \frac 12)$} denote the nonequispaced nodes.
Note that we changed time and frequency domain to match our notations in Section~\ref{sec:nfft}.
%
%
Thereby,
we receive the scalar product
\begin{equation*}
	\langle \varphi_j, \psi_l \rangle_{\ell_2} 
	=
	\sum_{k=-\infty}^{\infty} \varphi_j(k) \, \overline{\psi_l(k)}
	= 
	\sum_{k=-\infty}^{\infty} \frac{1}{M_\sigma \hat w(k)} \, \e^{-2\pi\i k \left(x_j - \frac{l}{M_{\sigma}}\right)}.
\end{equation*}
Truncating the infinite sum yields an approximation of the matrix~$\b \Phi$ in~\eqref{eq:matrix_Phi} by
\begin{equation}
\label{eq:matrix_Phi_l2}
	\b \Phi_{\ell_2} = \bigg( \overline{K \hspace{-2pt} \left(x_j -\tfrac{l}{M_{\sigma}}\right)} \bigg)_{j=1,\, l=-\frac{M_\sigma}{2}}^{N,\ \frac{M_\sigma}{2}-1}
\end{equation}
with the kernel~$K(x)$ from~\eqref{eq:kernel}. 
In the following explanations we choose \mbox{$\b \Phi = \b \Phi_{\ell_2}$}.

\begin{Remark}
In general, we do not have admissible frames for our known window functions~$w$ because of the factor 
\mbox{$\frac{1}{\hat w(k)}$}, \mbox{$k=-\infty,\dots,\infty$}.
Only for finite frames the appropriate conditions can be satisfied.
In addition, it must be pointed out that for other sampling patterns than the jittered equispaced nodes it was already mentioned in~\cite{GeSo16} that the admissibility condition may not hold or even the conditions for constituting a frame may fail, cf.~\cite{GeSo14}.
\ex
\end{Remark}

\subsubsection{Theoretical results\label{subsubsec:theory}}

For these given frames we consider again the frame approximation~\eqref{eq:approx_function}.
Our aim is to show that the inversion of the NFFT illustrated in Section~\ref{subsec:rect} can also be expressed by means of a frame-theoretical approach, i.\,e., by approximating a function~$\hat f$ in the frequency domain, cf.~\eqref{eq:fourier_transform}, and subsequently sampling at equispaced points \mbox{$k=-\frac M2, \dots, \frac M2-1$}.

The frame approximation of the function~$\hat f$ is given by
\begin{equation}
\label{eq:approx_function_fourier}
	\tilde{\hat f} 
	= 
	\sum_{j=1}^{N} \sum_{l=-\frac{M_\sigma}{2}}^{\frac{M_\sigma}{2}-1} 
	\langle \hat f, \varphi_j \rangle_{\ell_2} \, p_{l,j} \, \psi_l
\end{equation}
with $p_{l,j}$ as defined in~\eqref{eq:approx_dual}.
Hence, we are acquainted with two different methods to compute the Fourier coefficients~$\hat f_k$ from given data~\mbox{$\langle \hat f, \varphi_j \rangle \eqqcolon f_j$}, the frame approximation~\eqref{eq:approx_function_fourier} as well as the adjoint NFFT~\eqref{eq:approx_nfft*}.
In what follows, we suppose that we can achieve a reconstruction via frames.
Utilizing this, we modify the adjoint NFFT so that we can use this simple method to invert the NFFT.
Thus, we are looking for an approximation of the form 
\mbox{$\tilde h_k \approx \tilde{\hat f} (k) \approx \hat f_k$},\, \mbox{$k=-\tfrac M2, \dots, \tfrac M2-1$}.

To compare the adjoint NFFT and the frame approximation we firstly rewrite the approximation~\eqref{eq:approx_nfft*} of the adjoint NFFT by analogy with~\cite{GeSo14}.
This yields
\begin{equation}
\label{eq:frame_nfft*}
	\tilde h_k 
	= 
	\sum_{l=-\frac{M_{\sigma}}{2}}^{\frac{M_{\sigma}}{2}-1} c_l \, \psi_l(k), \quad
	k = -\tfrac M2,\dots,\tfrac M2-1,
\end{equation}
with coefficients vector
\begin{equation}
\label{eq:coeff_c}
	\b c 
	\coloneqq \left( c_l \right)_{l=-\frac{M_{\sigma}}{2}}^{\frac{M_{\sigma}}{2}-1}
	= \left( \sum_{j=1}^{N} f_j \, \tilde w_m \hspace{-2.75pt} \left(x_j -\tfrac{l}{M_{\sigma}}\right)\right)_{l=-\frac{M_{\sigma}}{2}}^{\frac{M_{\sigma}}{2}-1}
	= \b B^*\b f,
\end{equation}
where
\mbox{$\b f \coloneqq \left(f_j\right)_{j=1}^N = (\langle \hat f, \varphi_j \rangle_{\ell_2})_{j=1}^N.$}
Likewise we rewrite \eqref{eq:approx_function_fourier}, cf.~\cite{GeSo14}, as
\begin{equation}
\label{eq:frame_approx}
	\tilde{\tilde h}_k \coloneqq \tilde{\hat f} (k) 
	= 
	\sum_{l=-\frac{M_{\sigma}}{2}}^{\frac{M_{\sigma}}{2}-1} d_l \, \psi_l(k),
	\quad 
	k = -\tfrac M2, \dots, \tfrac M2-1,
\end{equation}
with
\mbox{$\b d \coloneqq \left( d_l \right)_{l=-\frac{M_{\sigma}}{2}}^{\frac{M_{\sigma}}{2}-1} = \b \Phi ^\dagger \b f.$}
Furthermore, we define the vectors
\mbox{$\tilde{\b h} \coloneqq (\tilde h_k)_{k=-\frac M2}^{\frac M2-1}$}
and
\mbox{$\tilde{\tilde{\b h}} \coloneqq (\tilde{\tilde h}_k)_{k=-\frac M2}^{\frac M2-1}$}
as well as the matrix
\mbox{$\b \Psi \coloneqq (\psi_l(k))_{k=-\frac M2,\, l=-\frac{M_\sigma}{2}}^{\frac M2-1,\ \frac{M_\sigma}{2}-1}$}.
Thereby, \eqref{eq:frame_nfft*} and \eqref{eq:frame_approx} can be represented by
\mbox{$\tilde{\b h} = \b \Psi \b c$} 
and 
\mbox{$\tilde{\tilde{\b h}} = \b \Psi \b d$}.
Hence, we can now estimate the difference between both approximations.

\begin{theorem}
Let 
\mbox{$\hat{\b w} \coloneqq \left( (\hat w(-k))^{-1} \right)_{k=-\frac M2}^{\frac M2-1}$}
be a vector satisfying \mbox{$\|\hat{\b w}\|_2 < \infty$}.
Then the following estimates hold.
\begin{enumerate}
	\item[(i)] 
	For \mbox{$M_\sigma < N$} we have
	\begin{equation}
	\label{eq:dist_N}
		\left\| \tilde {\b h} - \tilde{\tilde {\b h}} \right\|_{2} 
		\leq
		\frac 1{\sqrt{M_\sigma}} \left\| \hat{\b w} \right\|_{2} \,\|\b \Phi \b B^* - \b I_{N}\|_{\textrm F} \, \|\b \Phi^\dagger \b f\|_2.
	\end{equation}
	\item[(ii)] 
	For \mbox{$M_\sigma > N$} we have
	\begin{equation}
	\label{eq:dist_M}
		\left\| \tilde {\b h} - \tilde{\tilde {\b h}} \right\|_{2} 
		\leq
		\frac 1{\sqrt{M_\sigma}} \left\| \hat{\b w} \right\|_{2} \,\|\b B^* \b \Phi - \b I_{M_\sigma}\|_{\textrm F} \, \|\b \Phi^\dagger \b f\|_2,
	\end{equation}
\end{enumerate}
where $\b B^*$ denotes the adjoint matrix of~\eqref{eq:matrix_B} and \mbox{$\b \Phi = \b \Phi_{\ell_2}$} is given as in \eqref{eq:matrix_Phi_l2}.
\end{theorem}

\begin{proof}
By analogy with~\cite{GeSo14} Definitions~\eqref{eq:frame_nfft*} and \eqref{eq:frame_approx} 
imply
\begin{equation*}
	\begin{split}
	\left\| \tilde {\b h} - \tilde{\tilde {\b h}} \right\|_{2} 
	=
	\left\| \b \Psi \b c - \b \Psi \b d \right\|_{2} 
	\leq
	\left\| \b \Psi \right\|_{\mathrm F} \left\| \b c - \b d \right\|_{2} 
	=
	\sqrt{\sum_{k=-\frac M2}^{\frac M2-1} \sum_{l=-\frac{M_\sigma}{2}}^{\frac{M_\sigma}{2}-1} |\psi_l(k)|^2} \,\cdot
	\left\| \b c - \b d \right\|_{2} \\
	=
	\frac{1}{M_\sigma}
	\sqrt{\sum_{k=-\frac M2}^{\frac M2-1} \left| \frac{1}{\hat w(-k)} \right|^2
		\sum_{l=-\frac{M_\sigma}{2}}^{\frac{M_\sigma}{2}-1} 
		\underbrace{\left| \e^{-2\pi\i kl/M_{\sigma}} \right|^2}_{\leq 1}}\,\cdot
	\left\| \b c - \b d \right\|_{2} 
	\leq
	\frac 1{\sqrt{M_\sigma}} \left\| \hat{\b w} \right\|_{2} \left\| \b c - \b d \right\|_{2}.
	\end{split}
\end{equation*}
Next we consider the norm \mbox{$\|\b c - \b d\|_2$} separately.
\begin{enumerate}
	\item[(i)] 
	For \mbox{$M_\sigma < N$} we have by~\eqref{eq:coeff_c} and \eqref{eq:frame_approx} that
	\begin{equation*}
		\b c -\b d 
		=
		\big(\b B^* - \b \Phi ^\dagger \big) \b f
		=
		\big(\b B^* - (\b \Phi^* \b \Phi)^{-1}\b \Phi^* \big) \b f 
		=
		\big((\b \Phi^* \b \Phi)^{-1}\b \Phi^* \big) (\b \Phi \b B^* - \b I_N) \b f.
	\end{equation*}
	This leads to 
	\begin{equation*}
		\|\b c - \b d\|_2
		\leq
		\|\b \Phi \b B^* - \b I_N\|_{\textrm F} \,\|\big((\b \Phi^* \b \Phi)^{-1}\b \Phi^* \big) \b f\|_2 
		\leq
		\|\b \Phi \b B^* - \b I_N\|_{\textrm F} \, \|\b \Phi^\dagger \b f\|_2.
	\end{equation*}
	\item[(ii)] 
	In analogy, for \mbox{$M_\sigma > N$} we have that
	\begin{equation*}
		\b c -\b d 
		=
		\big(\b B^* - \b \Phi ^\dagger \big) \b f
		=
		\big(\b B^* - \b \Phi^*(\b \Phi \b \Phi^*)^{-1}\big) \b f 
		=
		(\b B^* \b \Phi - \b I_{M_\sigma}) \big(\b \Phi^*(\b \Phi \b \Phi^*)^{-1} \big) \b f
	\end{equation*}
	and thereby
	\begin{equation*}
		\|\b c - \b d\|_2
		\leq
		\|\b B^* \b \Phi - \b I_{M_\sigma}\|_{\textrm F} \,\|\big(\b \Phi^*(\b \Phi \b \Phi^*)^{-1} \big) \b f\|_2 
		\leq
		\|\b B^* \b \Phi - \b I_{M_\sigma}\|_{\textrm F} \, \|\b \Phi^\dagger \b f\|_2.
	\end{equation*}
\end{enumerate}
\rule{0pt}{0pt}
\end{proof}

\subsubsection{Optimization\label{subsubsec:opt}}

Our aim is to minimize the distances shown in~\eqref{eq:dist_N} and \eqref{eq:dist_M} to modify the adjoint NFFT such that we can achieve an inversion of the NFFT.
To this end, we suppose we are given nodes~$x_j$ as well as frames~$\{\varphi_j\}$ and~$\{\psi_l\}$ and thereby the matrix~\mbox{$\b \Phi=\b \Phi_{\ell_2}$}.
Thus, our purpose is to improve the approximation of the adjoint NFFT by modifying the matrix~$\b B^*$.

\paragraph{Connection to the first approach\label{par_opt_ansatz1}}

Firstly, we consider the case \mbox{$M_\sigma < N$}.
Minimizing the distance in~\eqref{eq:dist_N} yields the optimization problem
\begin{equation}
\label{eq:opt_Phi_B}
	\underset{\b B \in \R^{N\times M_\sigma} \colon \b B\, (2m+1)\text{-sparse }}{\text{Minimize }}\ \|\b \Phi \b B^* - \b I_N\|_{\textrm F}^2,
\end{equation}
which is of similar form to those seen in Section~\ref{subsec:rect}. 
For solving this problem we have a closer look at the matrix~\mbox{$\b \Phi \b B^*$}.
By Definitions~\eqref{eq:matrix_B} and \eqref{eq:matrix_Phi_l2} we obtain
\begin{equation}
\label{eq:matrix_Phi_B}
	\b \Phi \b B^*
	=
	\left[
	\sum_{l=-\frac{M_\sigma}{2}}^{\frac{M_\sigma}{2}-1} 
	\sum_{k=-\frac M2}^{\frac M2-1} 
	\frac{1}{M_\sigma \hat w(k)} \, \e^{-2\pi\i k \left(x_j-\frac{l}{M_\sigma}\right)} \, \tilde w_m \hspace{-2.75pt} \left(x_h-\tfrac{l}{M_{\sigma}}\right)
	\right]
	_{j,\,h=1}^{N}.
\end{equation}
In addition, we consider analogously to~\eqref{eq:matrix_B_Phi_T}
\begin{equation}
\label{eq:matrix_cp_T}
	\b B \b F \b D \b A^*
	=
	\left[
	\sum_{l=-\frac{M_\sigma}{2}}^{\frac{M_\sigma}{2}-1} 
	\sum_{k=-\frac M2}^{\frac M2-1} 
	\frac{1}{M_\sigma \hat w(k)} \, \e^{-2\pi\i k \left(x_j-\frac{l}{M_\sigma}\right)} \, \tilde w_m \hspace{-2.75pt} \left(x_h-\tfrac{l}{M_{\sigma}}\right)
	\right]
	_{h,\,j=1}^{N}.
\end{equation}
Comparing these matrices~\eqref{eq:matrix_Phi_B} and \eqref{eq:matrix_cp_T} we recognize that they are exactly the transposed of each other, i.\,e.,
\mbox{$
\b \Phi \b B^*
=
\left( \b B \b F \b D \b A^* \right)^T
=
\left( \b B \b K^* \right)^T.
$}
Thereby, \eqref{eq:opt_Phi_B} is equivalent to the problem
\begin{equation*}
\underset{\b B \in \R^{N\times M_\sigma} \colon \b B\, (2m+1)\text{-sparse }}{\text{Minimize }}\ \|\b B \b F \b D \b A^* - \b I_N\|_{\textrm F}^2
\end{equation*}
and can be solved like already seen in Section~\ref{subsubsec:ansatz1}.
It may be recognized that the
objective is a slightly different one since now we seek an approximation of the form 
\mbox{$\b B \b F \b D \b A^* \approx \b I_N$}
instead of
\mbox{$\b B \b F \b D \b A^* \approx M \b I_N$}.
However, the constant does not change the method and thus the same fast algorithm can be used.

\paragraph{Connection to the second approach\label{par_opt_ansatz2}}

For \mbox{$M_\sigma>N$} we consider the estimate~\eqref{eq:dist_M} where minimization leads to the optimization problem
\begin{equation}
\label{eq:opt_B_Phi}
	\underset{\b B \in \R^{N\times M_\sigma} \colon \b B\, (2m+1)\text{-sparse }}{\text{Minimize }}\ \|\b B^* \b \Phi - \b I_{M_\sigma}\|_{\textrm F}^2.
\end{equation}
Again we have a closer look at the appropriate matrix
\begin{equation}
\label{eq:matrix_B_Phi}
	\b B^* \b \Phi
	=
	\left[ 
	\sum_{j=1}^{N} 
	\sum_{k=-\frac M2}^{\frac M2-1} 
	\frac{1}{M_\sigma \hat w(k)} \, \e^{-2\pi\i k \left(x_j-\frac{s}{M_\sigma}\right)} \, \tilde w_m \hspace{-2.75pt} \left(x_j-\tfrac{l}{M_{\sigma}}\right)
	\right]
	_{l,\,s=-\frac{M_\sigma}{2}}
	^{\frac{M_\sigma}{2}-1 }
\end{equation}
and additionally consider the matrix from~\eqref{eq:matrix_B_Phi_T}.
%
%
Once more, a comparison of \eqref{eq:matrix_B_Phi} and \eqref{eq:matrix_B_Phi_T} yields 
\mbox{$
\b B^* \b \Phi
=
\left( \b F \b D \b A^* \b B \right)^T
=
\left( \b K^* \b B \right)^T,
$}
i.\,e., they are equal except for transposition.
Because \eqref{eq:opt_B_Phi} is hence equivalent to the transposed problem
\begin{equation*}
	\underset{\b B \in \R^{N\times M_\sigma} \colon \b B\, (2m+1)\text{-sparse }}{\text{Minimize }}\ \|\b F \b D \b A^* \b B - \b I_{M_\sigma}\|_{\textrm F}^2,
\end{equation*}
\eqref{eq:opt_B_Phi} can be solved like already discussed in Section~\ref{subsubsec:ansatz2}.

Therefore, we have shown that the frame-theoretical approach can be traced back to the methods for inverting the NFFT introduced in Section~\ref{subsec:rect}.
In other words, the explanations in Section~\ref{sec:frames} can be seen as simply having a different point of view to the problem of Section~\ref{subsec:rect}.

\begin{Remark}
Note that the method of~\cite{GeSo14} is based only on optimizing the diagonal matrix~$\b D$ whereas we used similar ideas to modify the sparse matrix~$\b B$.
\ex
\end{Remark}

\section{Conclusion}
In the present paper we developed new direct methods for computing an inverse NFFT, i.\,e., for the reconstruction of $M$ Fourier coefficients~$\hat f_k$ from given $N$ nonequispaced data~$f_j$.
Furthermore, solutions for the adjoint problem, the reconstruction of function values~$f_j$ from given data~$h_k$, were proposed.
For both problems we derived efficient algorithms for the quadratic setting as well as for the overdetermined and underdetermined case.

In the quadratic setting we used a relation between two evaluations of a trigonometric polynomial which can be deduced by means of Lagrange interpolation.
Ap\-pro\-xi\-ma\-tion of corresponding coefficients by means of the fast summation yields algorithms of complexity~$\mathcal O(N\log N)$.

The main idea for the overdetermined and underdetermined cases was the minimization of a certain Frobenius norm so that the solution can be deduced by means of the least squares method.
All in all, we ended up with precomputational algorithms of complexity \mbox{$\mathcal O(N^2)$} and~\mbox{$\mathcal O(M^2)$}, respectively, whereas the algorithms for the inversion require only \mbox{$\mathcal O(M\log M+N)$} arithmetic operations.

Finally, we investigated an approach based on~\cite{GeSo14} considering frame approximation which can be used to approximate a function~$\hat f$ in the frequency domain and subsequently sample at equispaced points.
By comparing this procedure to the adjoint NFFT we modified the last-mentioned to achieve an iNFFT.
In so doing, we found out that the thereby obtained approaches can be traced back to the methods for the inversion introduced for the overdetermined and underdetermined cases.

For the future it might be of interest to study for what kind of distribution of nodes and which window functions the frame-theoretical approach is applicable.
Moreover, a generalization of the presented methods to higher dimensions is subject of ongoing research.

\section*{Acknowledgments}
The first named author gratefully acknowledges the funding support from the European Union and the
Free State of Saxony (ESF).
Moreover, the authors thank the referees and the editor for their very useful suggestions for improvements.

\bibliographystyle{abbrv}

\end{document}